\documentclass[11pt,a4paper]{amsart}
\usepackage{amsfonts,amscd,amsmath,amssymb,amsthm} % Math and stuff
\usepackage{xparse} % Package for commands with optional arguments
\usepackage{dsfont} % Package for theorems, lemmas, corollaries, proofs and definitions
\usepackage{xcolor} % Package for sev­eral kinds of color tints, shades, tones, and mixes of ar­bi­trary col­ors
\usepackage[utf8]{inputenc} % Package for input in Latex, for it to be UTF-8 (standard format)
\usepackage{tikz} % Tikz for the win
\usepackage{MnSymbol} % http://ctan.org/pkg/mnsymbol for \udots

\numberwithin{equation}{section}

\newtheorem{theorem}{Theorem}[section]{\bf}{\it}
\newtheorem{proposition}[theorem]{Proposition}{\bf}{\it}
\newtheorem{corollary}[theorem]{Corollary}{\bf}{\it}
\newtheorem{lemma}[theorem]{Lemma}{\bf}{\it}
{\bf}{\it}
{\bf}{\it}
{\bf}{\it}
\newcounter{theoremintrocnt}
\newtheorem{introtheorem}{Theorem}{\bf}{\it}

\theoremstyle{definition}
\newtheorem{definition}[theorem]{Definition}%{\it}{\rm}
\newtheorem{remark}[theorem]{Remark}%{\it}{\rm}
\newtheorem{example}[theorem]{Example}%{\it}{\rm}
\newcommand{\quadtext}[1]{\quad \text{#1} \quad}

\title{The uniform Roe algebra of an inverse semigroup}

\author[Fernando Lled\'{o}]{Fernando Lled\'{o}$^{1}$}
\address{Department of Mathematics, University Carlos~III Madrid,
  Avda.~de la Universidad~30, 28911 Legan\'{e}s (Madrid), Spain
  and Instituto de Ciencias Matem\'{a}ticas (CSIC - UAM - UC3M - UCM).}
\email{flledo@math.uc3m.es}

\author[Diego Mart\'{i}nez]{Diego Mart\'{i}nez$^{2}$}
\address{Department of Mathematics, University Carlos~III Madrid,
  Avda.~de la Universidad~30, 28911 Legan\'{e}s (Madrid), Spain
  and Instituto de Ciencias Matem\'{a}ticas (CSIC - UAM - UC3M - UCM).}
\email{lumartin@math.uc3m.es}

\date{\today}
\thanks{{$^{1}$} Supported by research projects MTM2017-84098-P and Severo Ochoa SEV-2015-0554 of the Spanish Ministry of Economy and Competition
(MINECO), Spain.}
\thanks{{$^{2}$} Supported by research projects MTM2017-84098-P, Severo Ochoa SEV-2015-0554 and BES-2016-077968 of the Spanish Ministry of Economy and Competition
(MINECO), Spain.}

\begin{document}

  %----------------------------------------------------------------------------------------
  % TITLE, ABSTRACT AND TABLE OF CONTENTS
  %----------------------------------------------------------------------------------------
  \begin{abstract}
    Given a discrete and countable inverse semigroup $S$ one can study, in analogy to the group case, its geometric aspects. In particular, we can equip $S$ with a natural metric, given by the path metric in the disjoint union of its Sch\"{u}tzenberger graphs. This graph, which we denote by $\Lambda_S$, inherits much of the structure of $S$. In this article  we compare the C*-algebra $\mathcal{R}_S$, generated by the left regular representation of $S$ on $\ell^2(S)$ and $\ell^\infty(S)$, with the uniform Roe algebra over the metric space, namely $C^*_u(\Lambda_S)$. This yields a chacterization of when $\mathcal{R}_S = C^*_u(\Lambda_S)$, which generalizes finite generation of $S$. We have termed this by \textit{admitting a finite labeling} (or being FL), since it holds when $\Lambda_S$ can be labeled in a finitary manner.

    The graph $\Lambda_S$, and the FL condition, also allow to analyze large scale properties of $\Lambda_S$ and relate them with C*-properties of the uniform Roe algebra. In particular, we show that domain measurability of $S$ (a notion generalizing Day's definition of amenability of a semigroup, cf., \cite{ALM19}) is a quasi-isometric invariant of $\Lambda_S$. Moreover, we characterize property A of $\Lambda_S$ (or of its components) in terms of the nuclearity and exactness of the corresponding C*-algebras. We also treat the special classes of F-inverse and E-unitary inverse semigroups from this large scale point of view.
  \end{abstract}
  \maketitle
  \begin{center}
    \textit{Dedicated to Pere Ara on the occasion of his 60$\;^{th}$ birthday.}
  \end{center}
  %\tableofcontents

  %----------------------------------------------------------------------------------------
  % INTRODUCTION
  %----------------------------------------------------------------------------------------
  \section{Introduction}\label{sec_intro}
  Given a discrete and finitely generated group $G$ there are several C*-algebras naturally associated to it. Among them, one of the most relevant ones is the so-called \textit{uniform Roe algebra of $G$}, which can be constructed in, at least, the following ways (see, e.g., \cite{R03,BO08}):
  \begin{itemize}
    \item Let $\lambda \colon G \rightarrow \mathcal{B}(\ell^2(G))$ be the left regular representation of $G$ and consider $\ell^{\infty}(G)\subset\mathcal{B}(\ell^2(G))$ as diagonal operators. The algebra $\mathcal{R}_G$ is then the C$^*$-algebra generated by these operators, i.e.,
    \begin{equation}
      \mathcal{R}_G := C^*\big(\ell^{\infty}\left(G\right) \cdot \left\{\lambda_g \mid g \in G\right\}\big) \subset \mathcal{B}\left(\ell^2\left(G\right)\right). \nonumber
    \end{equation}
     Similarly, we can consider the natural action of $G$ on $\ell^{\infty}(G)$ given by left translation of the argument and construct the reduced crossed product
     $\ell^{\infty}(G) \rtimes_r G\subset \mathcal{B}(\ell^2(G))$.
    \item Given a finite, symmetric and generating set $K \subset G$, one may consider the (left) Cayley graph of $G$ with respect to $K$, which we denote by $\Lambda(G, K)$. Then $\Lambda(G, K)$ with the path distance is a metric space of \textit{bounded geometry} (see~\cite[Section~1.2]{NY12}) and, thus, one can construct the corresponding uniform Roe algebra $C_u^*(\Lambda(G, K))$, i.e., the C*-algebra generated by \textit{finite propagation} operators in $\mathcal{B}(\ell^2(G))$ (see Section~\ref{sec_mtrspc}). This C*-algebra is, in particular, generated by partial isometries representing partial bijections of the metric space $t \colon A \rightarrow B$, where $A, B \subset G$ and
    \begin{equation}\label{eq_prttrn_gr}
      \sup \left\{ d\left(a, t\left(a\right)\right) \, \mid \, a \in A \right\} < \infty.
    \end{equation}
  \end{itemize}
  The following well known theorem states that the approaches above generate the same C*-algebra (cf., \cite[Propositon~5.1.3]{BO08}).
  \begin{theorem}\label{gr_thm}
    Let $G$ be a countable and discrete group $G$. Then
    \begin{equation}
      \ell^{\infty}\left(G\right) \rtimes_r G = \mathcal{R}_G = C_u^*\left(\Lambda\left(G, K\right)\right). \nonumber
    \end{equation}
  \end{theorem}
  Each description above has its advantages. For instance, using the approximations provided by $\mathcal{R}_G$, where elements can be approximated by linear combinations of terms of the form $f \lambda_{g}$, where $g\in G$ and $f\in\ell^\infty(G)$, is convenient to study the trace space of the algebra (see, e.g., \cite{ALLW18R,ALM19}). On the other hand, crossed products are instead useful to analyze structural aspects of the algebra like, e.g., the relation between the nuclearity of the C*-algebra and the amenability of the action (cf., \cite[Theorem~4.3.4]{BO08}). Finally, the C*-algebra $C_u^*(\Lambda(G, K))$ captures the \textit{large-scale geometry} of $G$. Moreover, the preceding theorem shows that every partial translation $t \colon A \rightarrow B$ as in Eq.~(\ref{eq_prttrn_gr}) can also be represented as a combination of left multiplication by elements of the group $G$. This proves to be helpful when studying the \textit{property A} of $G$ (see~\cite{BO08,NY12,R03} or Section~\ref{sec_scht_a}).

  In the present article we will generalize Theorem~\ref{gr_thm} to the context of a discrete and countable inverse semigroup $S$. In particular, we will associate with $S$ a graph $\Lambda_S$, endowed with the path length metric, and focus on the relation between the C*-algebras $\mathcal{R}_S$ and $C^*_u(\Lambda_S)$. We also analyze large scale properties of $\Lambda_S$ and relate them with C*-properties of the uniform Roe algebra.

  Recall that a semigroup $S$ (which we assume to be countable and discrete) is a set equipped with a binary and associative operation. We say that $S$ is an \textit{inverse semigroup} if for every $s \in S$ there is a unique $s^* \in S$ such that $ss^*s = s$ and $s^*ss^* = s^*$. We denote by $E(S):=\{s^*s\mid s\in S\}$ the set of projections, which is a commutative inverse subsemigroup of $S$ (some standard references on this topic are \cite{H76,H95,L98}). The dynamics in $S$ given by left multiplication are locally injective, meaning that for each $s \in S$ there is some $D_{s^*s} \subset S$ (the domain of $s$) such that if $sx = sy$ then $x = y$ for any $x, y \in D_{s^*s}$. This allows to analyze notions like amenability (introduced by Day for general semigroups in \cite{D57}) in a more detailed way. For example, it was introduced in \cite{ALM19} the more general notion of {\em domain measurability}, which precisely captures the dynamical invariance of the probability measure on $S$ (cf., Section~\ref{sec_qiinv}). Moreover, inverse semigroups allow for a natural connection with C*-algebras via a standard generalization of the left regular representation mentioned above. The representing operators in this case are partial isometries denoted by $V_s$, where $s\in S$. The C*-algebra $\mathcal{R}_S$ is, in analogy to the group case, the C*-completion of the sets $\{V_s\mid s\in S\}$ and $\ell^\infty(S)$. We refer to~\cite{A16,M10,P12,DP85,E08,MS14} for additional references on the relation with groupoids and C*-algebras and to~\cite{F13,GK13} for connections to coarse geometry.

  To obtain a metric space associated to $S$ we have to consider the so-called (left) Sch\"utzenberger graphs of $S$. Recall that two elements $x,y\in S$ are \textit{$\mathcal{L}$-related} if $x^*x = y^*y$. The $\mathcal{L}$-classes of $S$ contain exactly one projection, and, hence, the semigroup decomposes into the disjoint union labeled by elements in $E(S)$:
  \[
   S=\sqcup_{e\in E(S)}L_e\;.
  \]
  Assuming that $S$ is generated by a fixed symmetric set $K$, each component above can be given a graph structure (and hence a path length metric), where $x, y \in L_e$ are joined by an edge if $kx = y$ for $k\in K$. The resulting graph is known as the (left) Sch\"{u}tzenberger graph of $e \in S$, which we denote by $\Lambda_{L_e}$. Their disjoint union $\Lambda_S:=\sqcup_{e}\Lambda_{L_e}$ is then an undirected graph that describes the geometry of the semigroup. Observe that $\Lambda_S$ can also be constructed by erasing all the directed edges in the usual left Cayley graph of $S$ (see Section~\ref{sec_dist}). This construction allows to construct the uniform Roe algebra $C^*_u(\Lambda_S)$, which is an analogue of the uniform Roe algebra of a group. However, in the context of inverse semigroups, the C*-algebra $\mathcal{R}_S$ needs not coincide with $C^*_u(\Lambda_S)$ (see Example~\ref{ex_c0}). The problem arises as the semigroup $S$ is naturally ordered, and the C*-algebra $\mathcal{R}_S$ inherits that order, while the graph $\Lambda_S$ does not. To circumvent this problem we introduce the notion of a \textit{finite labeling} of $S$ (see Definition~\ref{def_fl}), which generalizes the class of finitely generated inverse semigroups. In fact, this notion is not only required if one wants to generalize Theorem~\ref{gr_thm}, but it turns out that it is also necessary (cf., Theorem~\ref{thm_mtrspc}):
  \begin{introtheorem}\label{introthm_mtrspc}
    Let $S$ be a countable and discrete inverse semigroup generated by a symmetric set $K \subset S$. The following statements are equivalent:
    \begin{enumerate}
      \item \label{introthm_mtrspc_fl} The pair $(S, K)$ admits a finite labeling (see Definition~\ref{def_fl}). 
      \item \label{introthm_mtrspc_calg} The following C*-algebras are equal: $\mathcal{R}_S=\ell^{\infty}(S) \rtimes_r S = C^*_u(\Lambda(S, K))$.
    \end{enumerate}
  \end{introtheorem}
  
  In this article we also address several \textit{large scale properties} of the inverse semigroup in relation with, in particular, the graph $\Lambda_S$. For example, we show that domain measurability of $S$ is a quasi-isometric invariant of $\Lambda_S$ (cf., Theorem~\ref{ame_qiinv_gr}), generalizing the well known statement that amenability of groups is a quasi-isometric invariant.
  \begin{introtheorem}\label{introthm_NEW}
   Let $S$ and $T$ be quasi-isometric inverse semigroups. Suppose they both admit a finite labeling. If $T$ is domain measurable then so is $S$.
  \end{introtheorem}

  Finally, we characterize the property A of the components of $\Lambda_S$ in terms of the nuclearity and exactness of the corresponding C*-algebras (see Theorem~\ref{invsem_thm_l_a}). Indeed, since the connected components of the graph $\Lambda_S$ are the Schützenberger graphs of $S$ (see~\cite{Sch57,Sch58,L98} and Section~\ref{sec_mtrspc} below), we have the following:
  \begin{introtheorem}\label{introthm_a}
    Let $S$ be a countable and discrete inverse semigroup, and let $L \subset S$ be an $\mathcal{L}$-class such that the left Schützenberger graph $\Lambda_L$ is of bounded geometry. Let $p_L$ be the orthogonal projection from $\ell^2(S)$ onto $\ell^2(L)$. Suppose, moreover, that the graph $\Lambda_L$ admits a finite labeling. Then the following are equivalent:
    \begin{enumerate}
      \item The graph $\Lambda_L$ has Yu's property A.
      \item The C*-algebra $p_L \mathcal{R}_S p_L$ is nuclear.
      \item The C*-algebra $p_L C_r^*(S) p_L$ is exact.
    \end{enumerate}
  \end{introtheorem}

  The article is structured as follows. Section~\ref{sec_invsem} recalls definitions, notation and constructions in the context of inverse semigroups. Furthermore, in Section~\ref{sec_crossed} we prove that, for inverse semigroups, $\mathcal{R}_S = \ell^{\infty}\left(S\right) \rtimes_r S$. 
  Section~\ref{sec_mtrspc} considers the more subtle metric space scenario, constructs the graph $\Lambda_S$ and proves Theorem~\ref{introthm_mtrspc}. 
  Moreover, Proposition~\ref{prop:fl:newchr} gives a geometrical characterization of the finite labeling property of inverse semigroups (see also Remark~\ref{rem:fl:cover}).
  Section~\ref{sec_qiinv} studies two quasi-isometric invariants of the graph $\Lambda_S$, first domain measurability and then property A. In particular, Section~\ref{sec_scht_a} focuses on the left Sch\"{u}tzenberger graphs of $S$ and proves Theorem~\ref{introthm_a}. As an application, we also relate the exactness of the reduced semigroup C$^*$-algebra, the nuclearity of the uniform Roe algebra and the property A of the graph $\Lambda_S$. The special classes of F-inverse and E-unitary inverse semigroups are also considered.

  \textbf{Conventions:} Throughout the paper, $S$ stands for a countable and discrete inverse semigroup (not necessarily unital and not necessarily finitely generated). 
  We will use the symbol $K\subset S$ to denote a (countable) and symmetric set of generators, i.e., if $k\in K$ then $k^*\in K$. 
  We will denote by $\ell^2(S)$ the complex Hilbert space of square-summable functions $\phi \colon S \rightarrow \mathbb{C}$, and by $\mathcal{B}(\ell^2(S))$ the space of bounded operators on it. We canonically embed $\ell^{\infty}(S)$ into $\mathcal{B}(\ell^2(S))$ as diagonal operators. The canonical orthonormal basis of $\ell^2(S)$ will be denoted by $\{\delta_x\}_{x \in S}$. The norm of an operator $T \in \mathcal{B}(\ell^2(S))$ is denoted by $||T||$. Given two sets $X_1,X_2$ we denote their disjoint union by $X_1\sqcup X_2$ and the cardinality of $X_1$ by $|X_1|$.

  \textbf{Acknowledgements:} The second author would like to thank N\'{o}ra Szak\'{a}cs for fruitful conversations on Sch\"{u}tzenberger graphs and related topics. We also thank the referee for their helpful comments on a previous version of the manuscript.
  
  %----------------------------------------------------------------------------------------
  % INVERSE SEMIGROUPS
  %----------------------------------------------------------------------------------------
  \section{Inverse semigroups and C*-algebras}\label{sec_invsem}
  In this section we introduce the definition of inverse semigroup as well as some important structures and examples that will be needed later. Some standard textbooks for additional motivation and proofs are \cite{H76,H95,L98,P12}. Given an inverse semigroup $S$ there is a C*-algebra $\mathcal{R}_S$ which naturally generalizes the uniform Roe algebra of a discrete countable group (cf., \cite[Proposition~5.1.6]{BO08}). For results relating the amenability of $S$ and C*-properties of $\mathcal{R}_S$ we refer to~\cite{ALM19}.
  \begin{definition}
    An \textit{inverse semigroup} is a non-empty set $S$ equipped with an associative binary operation such that for all $s \in S$ there is a unique $s^* \in S$ satisfying $s s^* s = s$ and $s^* s s^* = s^*$.
  \end{definition}
  We say that $S$ is unital if there is an element $1 \in S$ such that $1 s = s 1 = s$ for all $s \in S$.
  An element $e \in S$ is a \textit{projection}, or \textit{idempotent}, if $e = e^2$. Observe that this implies that $e$ is self-adjoint, i.e., $e = e^*$. The set of projections is denoted by $E(S)$. Note that $s^*s$ is a projection for all $s \in S$, and thus $E(S)$ is never empty. Furthermore, $E(S)$ is a commutative inverse sub-semigroup of $S$ (see~\cite{V53} or~\cite[Theorem~3]{L98}). Any group is trivially an inverse semigroup with $s^*=s^{-1}$. Conversely, it is easy to see that $S$ is a group if and only if $E(S)$ has exactly one element, namely the identity of the group.

  We say that $s,t \in S$ are \textit{$\sigma$-equivalent} if there is some projection $e\in E(S)$ such that $se = te$. Note this is equivalent to the existence of some projection $f \in E(S)$ such that $fs = ft$ (take $f = ses^* = tet^* \in E(S)$). The relation $\sigma$ is a congruence in $S$ and we denote by $G(S) := S/\sigma$ the corresponding quotient, which is a group called the \textit{maximal homomorphic image} of $S$. For simplicity we will also denote by $\sigma$ the canonical projection 
  $\sigma \colon S \rightarrow G(S)$.

  Any inverse semigroup $S$ has a natural partial order. Indeed, we say $s \leq t$ if there is some projection $e \in E(S)$ such that $s = te$. Observe that, again, this is equivalent to the existence of an idempotent $f \in E(S)$ satisfying $s = ft$ (just take $f := tet^*$). From the preceding definitions we have that if $s \leq t$ then $s \sigma t$, so that the partial order $\leq$ restricts to a partial order within the $\sigma$-classes. 
  
  Inverse semigroups have a canonical representation as {\em partial bijections in $S$} and this dynamical picture will be useful in this article. In fact, given an element $s \in S$, we define the \textit{domain of $s$} by 
  \begin{equation}
   D_{s^*s}:=\{ x\in S \mid s^*sx = x\} =\{ x\in S \mid xx^*\leq s^*s\} = s^*s \cdot S \subset S\;. \nonumber
  \end{equation}
  It can be shown that left multiplication by $s$ is a bijection between $D_{s^*s}$ and $D_{ss^*}$, the so-called \textit{range of $s$}. Moreover, observe that if $s\leq t$ then $s^*s\leq t^*t$ and $D_{s^*s}\subset D_{t^*t}$.
  \begin{example}\label{ex_int_poly}
    For any $n \in \mathbb{N} \cup \left\{\infty\right\}$ the \textit{polycyclic monoid} $\mathcal{P}_n$ (see, e.g., \cite{NP70,L98}) is the inverse monoid given by the presentation:
    \begin{equation}
      \mathcal{P}_n := \left\{1\right\} \sqcup \left\langle \; a_1, \dots, a_n \;\; \mid \;\; a_i^* a_j = 
        \left\{ \begin{array}{lcc}
                  1 & \text{if} & i = j \\
                  0 &  & \text{otherwise}
                \end{array}
        \right.
      \right\rangle \sqcup \{0\} \;.\nonumber
    \end{equation}
    If $n = 1$, for instance, it can be shown that $\mathcal{P}_1 = \{a^i a^{*j} \mid i, j \in \mathbb{N}\} \sqcup \{0\}$, where $E(\mathcal{P}_1) = \{a^ia^{*i} \mid i \in \mathbb{N}\} \sqcup \{0\}$. Moreover, one can see that $D_{a^ia^{*i}} = \{a^p a^{*q} \mid p \geq i\} \sqcup \{0\}$.
  \end{example}
  
  Now we give the construction of the C*-algebra $\mathcal{R}_S$, which generalizes the uniform Roe algebra of a discrete group. Given an inverse semigroup $S$, consider its \textit{left regular representation}:
  \begin{equation}
    V \colon S \rightarrow \mathcal{B}\left(\ell^2\left(S\right)\right), \quadtext{where} V_s (\delta_x) := \left\{
                      \begin{array}{lcc}
                        \delta_{sx} & \text{if} & x \in D_{s^*s} \\
                        0 &  & \text{otherwise}\;.
                      \end{array}
                    \right. \nonumber
  \end{equation}
  $V$ is then a faithful representation of $S$ by partial isometries on $\ell^2\left(S\right)$ (see~\cite{V53} and~\cite[Proposition~2.1.4]{P12}). Note the condition $x \in D_{s^*s}$ guarantees that $V_s$ is bounded. Moreover, we consider $\ell^{\infty}(S)$ as multiplication (i.e., diagonal) operators in $\ell^2(S)$. The C*-algebra $\mathcal{R}_S$ is the norm completion of the *-algebra generated by the products $f V_s$, where $f \in \ell^{\infty}(S)$ and $s \in S$:
  \begin{equation}
    \mathcal{R}_S := C^* \big(\ell^{\infty}\left(S\right) \, \cdot \, \left\{V_s\right\}_{s \in S}\big) \subset \mathcal{B}\left(\ell^2\left(S\right)\right). \nonumber
  \end{equation}
  Note that, in case $S$ is unital, the C*-algebra $\mathcal{R}_S$ will be generated by $\{V_s\}_{s \in S}$ and $\ell^{\infty}(S)$. However, if $S$ is not unital then $\mathcal{R}_S$ may or not be unital (see Example~\ref{ex_c0}). Another C*-algebra of interest to us is the so-called \textit{reduced semigroup C*-algebra}:
  \begin{equation}
    C_r^*\left(S\right) := C^*\big(\left\{V_s\right\}_{s \in S}\big) \subset \mathcal{R}_S. \nonumber
  \end{equation}

    %----------------------------------------------------------------------------------------
    % Loo AND R_S
    %----------------------------------------------------------------------------------------
    \subsection{The reduced crossed product}\label{sec_crossed}
    Given a C*-algebra $\mathcal{A}$ and a (continuous) action of a group $G$ by *-automorphisms on $\mathcal{A}$ the crossed product of $\mathcal{A}$ by $G$ is a larger C*-algebra containing $\mathcal{A}$ and a copy of $G$ as unitaries that implement the action (see, e.g.,~\cite[Section~4.1]{BO08}). This construction was later generalized to the setting of inverse semigroups (see, for instance,~\cite{S97,E08,MS14}). Since in the subsequent sections we will only deal with the commutative case $\mathcal{A} = C_0(X)$, that is the only case we introduce.
    \begin{definition}\label{def_act_invsem}
      Let $X$ be a locally compact Hausdorff space and $S$ be a countable inverse semigroup.
      \begin{enumerate}
        \item \label{def_act_invsem_partaut} A \textit{partial automorphism} of $C_0(X)$ is a triple $(\phi, E_1, E_2)$, where $E_i \triangleleft C_0(X)$ are closed two-sided ideals and $\phi \colon E_1 \rightarrow E_2$ is a *-isomorphism. The set of partial automorphisms of $C_0(X)$ is denoted by $\text{PAut}(X)$. We equip $\text{PAut}(X)$ with the binary operation given by composition of maps whenever defined.
        \item \label{def_act_invsem_act} An \textit{action} of $S$ on $X$ is a homomorphism $\alpha \colon S \rightarrow \text{PAut}(X)$, where 
        $s \mapsto (\alpha_s, E_{1,s}, E_{2,s})$.
        \end{enumerate}
    \end{definition}

    Observe that, for the action defined in~(\ref{def_act_invsem_act}) above, the domain of the map $\alpha_s$ (namely $E_{1,s}$) only depends on $s^*s$, and its range (namely $E_{2,s}$) only depends on $ss^*$. Indeed, this follows from the fact that $\alpha$ is a semigroup homomorphism and $s^*s$ acts as the identity on $D_{s^*s}$. Therefore, and for the sake of simplicity, we will henceforth denote $E_{1,s}$ by $E_{s^*s}$ and $E_{2,s}$ by $E_{ss^*}$. We will be particularly interested in the following action.
    \begin{proposition}\label{prop:act}
      Let $S$ be a countable and discrete inverse semigroup. Given $s \in S$ let
      \begin{equation}
        E_{s^*s} := \left\{f \in \ell^{\infty}\left(S\right) \;\, \text{such that} \;\, \text{supp}\,(f) \subset D_{s^*s} \right\}. \nonumber
      \end{equation}
      Then the map $\alpha \colon s \mapsto \alpha_s$ given by $(\alpha_s f)(x) = f(s^*x)$, where $f \in E_{s^*s}$, defines an action.
    \end{proposition}
    \begin{proof}
      The sets $E_{s^*s}$ are clearly (two-sided) closed ideals in $\ell^{\infty}(S)$. Furthermore, since left multiplication by $s^*$ defines a bijection from $D_{ss^*}$ onto $D_{s^*s}$ it follows that $\alpha_s \colon E_{s^*s} \rightarrow E_{ss^*}$ is a *-isomorphism. Moreover it is routine to show that
      \begin{equation}
        \alpha_s^{-1}\left(E_{ss^*} \cap E_{t^*t}\right) = E_{t^*s^*st}. \nonumber
      \end{equation}
      Finally, observe
      \begin{equation}
        \alpha_{ts} f \left(x\right) = f\left(s^*t^*x\right) = \left(\alpha_sf\right)\left(t^*x\right) = \alpha_t\left(\alpha_sf\right)\left(x\right) \nonumber
      \end{equation}
      and, thus, $\alpha_{st} f = \alpha_s \left(\alpha_t f\right)$ for any $f \in E_{t^*s^*st}$.
    \end{proof}

    Since no other action will be considered for crossed products we will denote $\alpha_s f$ simply by $sf$. In order to construct the \textit{reduced crossed product of $\ell^{\infty}(S)$ by $S$} recall that the canonical representation of $\ell^{\infty}(S)$ as multiplication operators in $\ell^2(S)$ is faithful, and consider
    \begin{equation}
      \pi \colon \ell^{\infty}\left(S\right) \rightarrow \mathcal{B}\left(\ell^2\left(S\right) \otimes \ell^2\left(S\right)\right), \quad \big(\pi\left(f\right)\big)\left(\delta_x \otimes \delta_y\right) :=
      \left\{ \begin{array}{lcc}
        f\left(yx\right) \delta_x \otimes \delta_y & \text{if} & x \in D_{y^*y}, \\
        0 &  & \text{otherwise}, 
      \end{array} \right. \nonumber
    \end{equation}
    and
    \begin{equation}
      1 \otimes V \colon S \rightarrow \mathcal{B}\left(\ell^2\left(S\right) \otimes \ell^2\left(S\right)\right), \quad \left(1 \otimes V_s\right) \left(\delta_x \otimes \delta_y\right) := \left\{ \begin{array}{lcc} \delta_x \otimes \delta_{sy} & \text{if} & y \in D_{s^*s}, \\ 0 &  & \text{otherwise,} \end{array} \right. \nonumber
    \end{equation}
    where $\{\delta_x\}_{x \in S}$ denotes the canonical orthogonal basis of $\ell^2(S)$. It follows from straightforward computations that $\pi$ and $1 \otimes V$ are faithful *-representations of $\ell^{\infty}(S)$ and $S$, respectively. Observe the representations intertwine the action in the following covariant way
    \begin{equation}
      \left(1 \otimes V_s\right) \pi\left(f\right) \left(1 \otimes V_s\right)^* = \pi\left(sf\right) \nonumber
    \end{equation}
    for all $s \in S$ and $f \in E_{s^*s}$. The \textit{reduced crossed product} $\ell^{\infty}(S) \rtimes_r S$ is then the C*-algebra generated by the images of $\pi$ and $1 \otimes V$, that is:
    \begin{equation}
      \ell^{\infty}\left(S\right) \rtimes_r S := C^*\Big(\pi\left(\ell^{\infty}\left(S\right)\right) \cdot \left\{1 \otimes V_s \mid s \in S\right\}\Big) \subset \mathcal{B}\left(\ell^2\left(S\right) \otimes \ell^2\left(S\right)\right). \nonumber
    \end{equation}
    The following result relates this construction to the algebra $\mathcal{R}_S$, and generalizes a standard result for groups (see~\cite[Proposition~5.1.3]{BO08}).
    \begin{theorem}
      Let $S$ be a countable and discrete inverse semigroup and consider the action of $S$ on $\ell^{\infty}(S)$ defined in Proposition~\ref{prop:act}. Then the C*-algebra $\mathcal{R}_S$ and the reduced crossed product are isomorphic, i.e., 
      \begin{equation}
       \mathcal{R}_S \cong \ell^{\infty}\left(S\right) \rtimes_r S \,. \nonumber
      \end{equation}
    \end{theorem}
    \begin{proof}
      Consider the bounded linear operator $W \colon \ell^2\left(S\right) \otimes \ell^2\left(S\right) \rightarrow \ell^2\left(S\right) \otimes \ell^2\left(S\right)$ given by
      \begin{equation}
        W \left(\delta_x \otimes \delta_y\right) = \left\{
                  \begin{array}{lcc}
                    \delta_x \otimes \delta_{yx} & \text{if} & xx^* = y^*y \nonumber \\
                    0 &  & \text{otherwise}.
                  \end{array} \right.
      \end{equation}
      It can be checked that $W$ is a partial isometry whose adjoint is
      \begin{equation}
        W^* \left(\delta_u \otimes \delta_v\right) = \left\{
                  \begin{array}{lcc}
                    \delta_u \otimes \delta_{vu^*} & \text{if} & u^*u = v^*v \nonumber \\
                    0 &  & \text{otherwise}.
                  \end{array} \right.
      \end{equation}
      Moreover, its initial projection $W^*W$ is the orthogonal projection onto the subspace generated by $\delta_x \otimes \delta_y$ where $xx^* = y^*y$, and its final projection $WW^*$ projects onto the subspace generated by $\delta_u \otimes \delta_v$ where $u^*u = v^*v$. In addition, it is routine to check that
      \begin{equation}
        W \pi\left(f\right) W^* = \left(1 \otimes f\right) WW^* = WW^* \left(1 \otimes f\right) \quad \text{and} \quad W \left(1 \otimes V_s\right) = \left(1 \otimes V_s\right) W. \nonumber
      \end{equation}
      It follows that the map $\text{Ad}(W)$ restricts to a *-isomorphism between $\ell^{\infty}(S) \rtimes_r S$ and $(1 \otimes \mathcal{R}_S) \subset 1 \otimes \mathcal{B}(\ell^2(S))$. Indeed, from the commutation relations above we have that
      \begin{equation}
        W \big(\ell^{\infty}\left(S\right) \rtimes_r S\big) W^* = {\text{cl}}_{\|\cdot\|} \left( \text{span} \Big\{1 \otimes f V_s \; \mid \; f \in \ell^{\infty}\left(S\right) \; \text{and} \; s \in S\Big\} \right) \, \cdot \, WW^* = \big(1 \otimes \mathcal{R}_S\big) WW^*, \nonumber
      \end{equation}
      which, in turn, is $*$-isomorphic to $1 \otimes \mathcal{R}_S$.
    \end{proof}

  %----------------------------------------------------------------------------------------
  % INVERSE SEMIGROUPS AS METRIC SPACES
  %----------------------------------------------------------------------------------------
  \section{Inverse semigroups and graphs}\label{sec_mtrspc}
  In this section we will give a third characterization of $\mathcal{R}_S$ as a uniform Roe algebra over a metric space naturally associated to a countable inverse semigroup $S$. Concretely, we will show in Theorem~\ref{thm_mtrspc} when $\mathcal{R}_S \cong C_u^*(\Lambda_S)$, where $\Lambda_S$ is an undirected graph (to be defined below) endowed with the path metric.
  
  We begin recalling the construction and properties of the uniform Roe algebra $C_u^*(X, d)$ of a metric space $(X, d)$. We refer to \cite{R03,BO08,NY12,SW13,BFV20} for proofs and additional motivation in the context of metric spaces. In \cite{ALLW18R} this class of algebras was generalized to so-called \textit{extended metric spaces} $(X, d)$ where the metric $d \colon X \times X \to [0,\infty]$ is allowed to take the value  $\infty$ (see~\cite[Section~2.1]{ALLW18}). This generalization is crucial for the main result of this section as $\Lambda_S$ splits into a disjoint union of connected components that, necessarily, are pairwise at infinite distance (see Proposition~\ref{pro:extended-needed} below).

  An extended metric space $(X, d)$ is of \textit{bounded geometry} if for every radius $R > 0$ the number of points within the balls of radius $R$ is uniformly bounded, i.e., for every $R > 0$ we have $\sup_{x \in X} |B_R(x)| < \infty$. Given an extended metric space of bounded geometry $(X, d)$, the {\em propagation} of a bounded and linear operator $T \in \mathcal{B}(\ell^2(X))$ is defined by
  \begin{align*}
    p(T):=\sup\Big\{d(x,y) \mid x,y\in X \;\; \text{and} \;\; T_{y, x} = \langle \delta_y, T \delta_x \rangle \neq 0\Big\}
  \end{align*}
  and $T$ has \textit{bounded} propagation if $p(T)<\infty$. The uniform Roe algebra $C_u^*(X, d)$ is the C*-algebra generated by the *-algebra $C_{\mathrm{u,alg}}^*(X, d)$ of operators with bounded propagation. The next proposition shows the need to consider extended metrics on $S$ if one wants to realize the equality $\mathcal{R}_S \cong C_u^*(\Lambda_S)$. In fact, the following result suggests that any pair $x, y \in S$ such that $x^*x \neq y^*y$ must be at infinite distance.
  \begin{proposition}\label{pro:extended-needed}
    Let $S$ be a countable inverse semigroup. Consider the matrix units 
    \begin{equation}
      M_{x, y} \colon \ell^2\left(S\right) \rightarrow \ell^2\left(S\right), \quad \text{where} \;\; x, y \in S \;\; \text{and} \;\;
        M_{x,y}\left(\delta_z\right) :=
          \left\{
            \begin{array}{lc}
              \delta_y & \text{if} \;\; z = x \\
              0 & \text{otherwise}.
            \end{array}
          \right. \nonumber
    \end{equation}
    Let $d \colon S \rightarrow [0, \infty]$ be an extended metric on $S$. For any $x, y \in S$ the following hold:
    \begin{enumerate}
      \item \label{prop:ext_metric:roe} If $d(x, y) < \infty$ then $M_{x, y} \in C_u^*(S, d)$.
      \item \label{prop:ext_metric:rs} If $x^*x \neq y^*y$ then $M_{x, y} \not\in \mathcal{R}_S$.
    \end{enumerate}
  \end{proposition}
  \begin{proof}
   (\ref{prop:ext_metric:roe}) follows from the fact that $M_{x, y} \in C_{\mathrm{u,alg}}^*(S, d) \subset C_u^*(S, d)$. For (\ref{prop:ext_metric:rs}), let $x, y \in S$ be such that $x^*x \neq y^*y$. Then the matrix unit $M_{x, y}$ is uniformly bounded away from any linear combination $\sum_{i = 1}^n f_i V_{s_i} \in \mathcal{R}_{S,\mathrm{alg}}$:
    \begin{align}
      \left|\left| M_{x, y} - \sum_{i = 1}^n f_i V_{s_i} \right|\right|^2 & \geq \left|\left| M_{x, y} \left(\delta_{x}\right) - \left(\sum_{i = 1}^n f_i V_{s_i}\right) \left(\delta_{x}\right) \right|\right|^2_{\ell^2\left(S\right)} \nonumber \\
      & = \left|\left|\delta_{y} - \sum_{\substack{i = 1 \\ x \in D_{s_i^*s_i}}}^n f\left(s_ix\right) \delta_{s_ix} \right|\right|^2_{\ell^2\left(S\right)} 
        = 1 + \left|\left|\sum_{\substack{i = 1 \\ x \in D_{s_i^*s_i}}}^n f\left(s_ix\right) \delta_{s_ix}\right|\right|^2_{\ell^2\left(S\right)}    \; \geq \; 1 \,, \nonumber
    \end{align}
    where the second equality follows from the fact that there can be no $s_i$ such that $x \in D_{s_i^*s_i}$ and $s_i x = y$, because otherwise $y^*y = x^* s_i^* s_i x = x^*x$, contradicting the hypothesis.
  \end{proof}

    %----------------------------------------------------------------------------------------
    % DISTANCES AND SCHUTZENBERGER GRAPHS
    %----------------------------------------------------------------------------------------
    \subsection{Sch\"{u}tzenberger graphs}\label{sec_dist}
     We will construct in this section a graph $\Lambda_S$, which will be disconnected unless $S$ is a group, that inherits the geometry of the semigroup. The connected components of $\Lambda_S$ are the so-called \textit{left Sch\"{u}tzenberger graphs} associated to each \textit{$\mathcal{L}$-class} of $S$. We begin introducing these notions.

    Consider a fixed symmetric generating set $K = K^* \subset S$  and note that, in general, we do not assume $K$ to be finite. Recall some of the so-called Green's equivalence relations (see, for instance,~\cite{G51,H95,L98}). For $x, y \in S$, we write $x \mathcal{L} y$ if $x^*x = y^*y$. Similarly, we say $x \mathcal{R} y$ if $xx^* = yy^*$. Note that, by definition, in each $\mathcal{L}$-class $L$ there is exactly one projection, namely $x^*x$ for any $x \in L$. Thus we can use the set $E(S)$ to label each class:
    \begin{equation}
      S = \sqcup_{e\in E(S)} L_e \quad \text{where for each $\mathcal{L}$-class we have} \; L_e\cap E(S)=\{e\} \,. \nonumber
    \end{equation}
    The next simple result will be used throughout the rest of the paper without explicit mention.
    \begin{lemma}\label{lemma_action_lclasses}
      Let $S$ be an inverse semigroup and let $s \in S$. Then $x \in D_{s^*s}$ if and only if $x \mathcal{L} sx$.
    \end{lemma}
    \begin{proof}
      Assume $x \mathcal{L} sx$. Then $x^*x=x^*s^*sx$ and, hence, multiplying by $x$ from the left we obtain $x=xx^*x=xx^*s^*sx=s^*sx$ which shows $x \in D_{s^*s}$. The reverse is done similarly.
    \end{proof}
    We introduce next the notion of Sch\"utzenberger graph which will play a central role in the rest of the article (see, e.g.,~\cite{Sch57,Sch58,GK13}).
    \begin{definition}
    Let $L$ be an $\mathcal{L}$-class of $S$. The \textit{left Sch\"{u}tzenberger graph} of $L$ is the edge-labeled undirected graph $\Lambda(L, K)$, where $K$ is a fixed symmetric generating set of $S$. Its vertex set is $L$ and two vertices $x, y \in L$ are joined by an edge labeled by $k \in K$ if $kx = y$. 
    We will denote by $\Lambda(S, K)$ the disjoint union of the left Sch\"{u}tzenberger graphs $\Lambda(L_e, K)$, where $e\in E(S)$. That is, the vertex set of $\Lambda(S, K)$ is $S$ and two vertices $x, y \in S$ are joined by an edge labeled by $k \in K$ if and only if $x \mathcal{L} y$ and $kx = y$. 
    \end{definition} 
    
    \begin{remark}
      Observe that for inverse semigroups $\Lambda(L, K)$ is an undirected graph, in the sense that if $x, y \in L$ and $kx = y$ then $y^*y = x^*k^*kx$ and, thus, $k^*y = k^*kx = x$. Therefore there is a $k^*$-labeled edge going from $y$ to $x$.
      Note also that $\Lambda(S, K)$ is not the usual (left) Cayley graph of a semigroup (see, e.g.,~\cite{GK13}). Indeed, observe that the Cayley graph of $S$ is in general a directed graph while $\Lambda(S, K)$ is always undirected. For instance, if $S = \{0, 1\}$ using the product as operation, then the Cayley graph of $S$ with $K=S$ has a directed edge going from $1$ to $0$, while in $\Lambda(S, K)$ the vertices $0$ and $1$ are in different connected components. In fact, from Lemma~\ref{lemma_action_lclasses} it is routine to show that $\Lambda(S, K)$ is the graph resulting from deleting the directed edges in the left Cayley graph of $S$.
    \end{remark}

    Alternatively one could construct the undirected graph $\Sigma(S, K)$, whose connected components are the \textit{right Sch\"{u}tzenberger graphs} $\Sigma(R, K)$ of each $\mathcal{R}$-class $R$. For general semigroups the left and right versions of these graphs need not even be coarsely equivalent, since an arbitrary semigroup could, for instance, have a distinct number of $\mathcal{L}$ and $\mathcal{R}$ classes (cf.,~\cite[Example~1]{GK13}). However, for inverse semigroups these graphs are isomorphic.
    \begin{lemma}\label{lemma_l_classes}
      The graphs $\Lambda(S,K)$ and $\Sigma(S,K)$ are isomorphic (as graphs).
    \end{lemma}
    \begin{proof}
      The isomorphism of graphs is given by the involution map, which is clearly bijective and maps $\mathcal{L}$-classes onto $\mathcal{R}$-classes. In addition, since the generator set is symmetric, we have that adjacency relations of the graphs are also preserved, since $y = kx$ if and only if $y^* = x^* k^*$.
    \end{proof}

    The preceding lemma justifies that we only need to consider left Sch\"{u}tzenberger graphs. For simplicity, we will omit the term \textit{left} from now on. Furthermore, if the generating set $K$ is clear from the context we will denote $\Lambda(S,K)$ just as $\Lambda_S$. Similarly, for any $\mathcal{L}$-class $L \subset S$ we will often write its Sch\"{u}tzenberger graph simply by $\Lambda_L$. We will also refer in some situations (specially in Section~\ref{sec_qiinv}) to the inverse semigroup as a metric space, by which we mean the corresponding disjoint union of Sch\"utzenberger graphs with the usual path length metric.
    \begin{remark}\label{rm_infedges}
      Observe that, contrary to the group case, infinitely many edges might connect two vertices $x, y \in S$. Indeed, let $S := G \times \mathbb{N}$, where $G$ is a discrete and countable group with at least two elements and generated by $K$. The semigroup operation is given by
      \begin{equation}
        \left(g, n\right) \cdot \left(h, m\right) := \left(gh, \min\left\{n, m\right\}\right). \nonumber
      \end{equation}
      Then $S$ is an inverse semigroup generated by $K\times\mathbb{N}$ with $(g,n)^*:=(g^{-1},n)$. Any $(g, 1), (kg, 1) \in G \times \{1\}$ are connected by infinitely many edges of the form $(k, m) \in K \times \mathbb{N}$. Likewise, any pair $(g, p), (kg, p) \in S$ are connected by infinitely many edges, labeled by $(k, q)$, where $q \geq p$.
    \end{remark}

    \begin{example}\label{ex_dist_poly}
      Consider the bicyclic semigroup $\mathcal{P}_1 = \langle a, a^* \mid a^*a = 1\rangle \sqcup \{0\}$ (see Example~\ref{ex_int_poly}). Fix the canonical symmetric generating set 
      $K := \{a, a^*, 0\}$ and note that any non-zero elements of $\mathcal{P}_1$ are of the form $a^i {a^*}^j$, for some $i, j \in \mathbb{N}$, of which exactly those of the form $a^i{a^*}^i$ are idempotents. Given $i, j, p, q \in \mathbb{N}$ observe that $a^i {a^*}^j \mathcal{L} a^p {a^*}^q$ if and only if $j = q$ and, thus, the map
      \begin{equation}
        \phi_j \colon \left\{a^i {a^*}^j \; \mid \; i \in \mathbb{N}\right\} \rightarrow \mathbb{N}, \quadtext{with} \phi_j\Big(a^i {a^*}^j\Big) := i \;.\nonumber
      \end{equation}
      is a bijection. The graph $\Lambda_{\mathcal{P}_1}$ is the disjoint union of copies of the usual Cayley graph of $\mathbb{N}$ (see Figure~\ref{fig_p1}), with an extra isolated component corresponding to the element $0 \in \mathcal{P}_1$. In general, the left Sch\"{u}tzenberger graph of an $\mathcal{L}$-class of $\mathcal{P}_n$ is the $n$-ary complete rooted tree, unless it is the $\mathcal{L}$-class of $0$, in which case the component is an isolated point.
      \begin{figure}[ht]
        \centering
        \begin{tikzpicture}
          \tikzset{invsem/.style={circle,minimum size=20pt}}
          \def \base {4} % Must be larger than 1
          \def \height {4} % Must be larger than 1
          \def \margin {0.25}

          % Draw the lower most line and second lower most
          \node [invsem] (0:0) at (0,0) {$1$};
          \node [invsem] (0:1) at (1,0) {$a^*$};
          \foreach \j in {2,...,\base} { \node [invsem] (0:\j) at (\j,0) {${a^*}^\j$}; }
          \node [invsem] (1:0) at (0,1) {$a$};
          \node [invsem] (1:1) at (1,1) {$aa^*$};
          \foreach \j in {2,...,\base} { \node [invsem] (1:\j) at (\j,1) {$a{a^*}^\j$}; }
          
          % Draw the left most line and second left most line
          \foreach \i in {2,...,\height} { \node [invsem] (\i:0) at (0,\i) {$a^\i$}; }
          \foreach \i in {2,...,\height} { \node [invsem] (\i:1) at (1,\i) {$a^\i a^*$}; }

          % Draw the rest of the vertices
          \foreach \i in {2,...,\height} { \foreach \j in {2,...,\base} { \node [invsem] (\i:\j) at (\j,\i) {$a^\i {a^*}^\j$}; } }

          % Draw the vertices with dots
          \foreach \j in {0,...,\base} { \node (dots:\j) at (\j,\base+1) {$\vdots$}; }
          \foreach \i in {0,...,\height} { \node [invsem] (\i:dots) at (\base+1,\i) {\dots}; }
          \node [invsem] (ddots) at (\base+1,\height+1) {$\udots$};
          \node [invsem] (zero) at (-1,0) {$0$};

          % Now we draw all the edges
          \foreach \j in {0,...,\base}
          {
            \foreach \i in {1,...,\height} { \draw (\j,\i-1+\margin) -- (\j,\i-\margin); }
            \draw (\j,\height+\margin) -- (\j,\height+1-\margin);
          }
        \end{tikzpicture}
        \caption{Left-Schützenberger graphs of the semigroup $\mathcal{P}_1 = \langle a, a^* \mid a^*a = 1\rangle \sqcup \left\{0\right\}$. Any two lines are at infinite distance from each other.}
        \label{fig_p1}
      \end{figure}
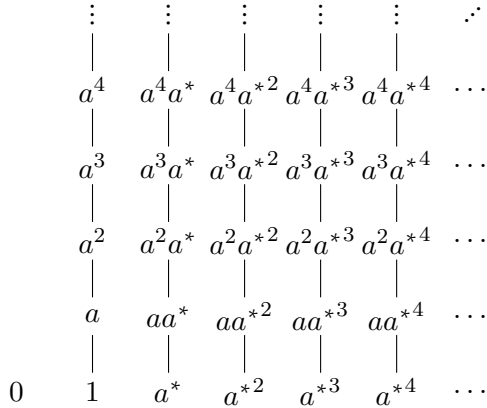
    \end{example}

    As usual, we consider the graph $\Lambda_S=\sqcup_{e\in E(S)}\Lambda_{L_e}$, where $\Lambda_{L_e}$ is equipped with the path distance between the vertices. We complete this section stating some facts about the graph $\Lambda_S$.
    \begin{proposition}\label{prop_dist_bndgeo}
      Let $S$ be a finitely generated inverse semigroup. Then $\Lambda_S$ is of bounded geometry.
    \end{proposition}
    \begin{proof}
      Denoting by $K$ the finite and symmetric generating set we have the uniform estimate $|B_R(x)| \leq |K|^R$.
    \end{proof}

    Our setting is that of semigroups which are countable but not necessarily finitely generated. We do, however, require $\Lambda_S$ to be of bounded geometry, which is an important condition to define the uniform Roe algebra $C_u^*(\Lambda_S)$ and study property A in the following section. The next result shows that our setting reduces to a common one in the case of groups. Recall that the Cayley graph of a group $G$ is bounded geometry if and only if $G$ is finitely generated.
    \begin{proposition}\label{prop_bddgeom}
      Let $S$ be an inverse semigroup. If $\Lambda_S$ is of bounded geometry, then so is the left Cayley graph of the maximal homomorphic image $G(S)$.
    \end{proposition}
    \begin{proof}
      Given $R > 0$, since $\Lambda_S$ has bounded geometry, we have that
      \begin{equation}
       m := \sup_{x \in S} \left| B_R(x) \right| < \infty \,. \nonumber
      \end{equation}
      We claim that $m$ also uniformly bounds the cardinality of the $R$-balls in the Cayley graph of $G(S)$, thus proving that $G(S)$ has bounded geometry. Indeed, assume that this is not the case. Consider $\sigma \colon S \rightarrow G(S)$ and let $B_R(\sigma(x_1))$ be an $R$-neighborhood (with respect to the path distance) of some $\sigma(x_1) \in G(S)$ having at least $m+1$ different points, i.e.,
      \begin{equation}
        \left\{\sigma\left(x_1\right), \dots, \sigma\left(x_{m + 1}\right)\right\} \subset B_R\left(\sigma\left(x_1\right)\right). \nonumber
      \end{equation}
      Any point $\sigma(x_i)$ is connected with $\sigma(x_1)$ by a geodesic path $\sigma(s_i) \in G(S)$ of length at most $R$, and in particular $\sigma(x_i) = \sigma(s_i x_1)$ for every $i = 2, \dots, m+1$. Consider then the idempotent given by
      \begin{equation}
        e \, := \, x_1^* s_2^*s_2 x_1 \, \dots \, x_1^* s_{m+1}^* s_{m+1} x_1, \nonumber
      \end{equation}
      and let $y_1 := x_1 e$ and $y_i := s_i y_1$, where $i = 2, \dots, m+1$. We claim that the points $\{y_i\}_{i = 1}^{m + 1}$ are pairwise different and within distance $R$ of $y_1$, thus proving that $m$ is not a bound on the cardinality of the $R$-balls of $S$ and contradicting the hypothesis. Indeed, first note that, when $i \neq j$, it follows that $y_i \neq y_j$ since $\sigma(y_i) = \sigma(s_ix_1) = \sigma(x_i) \neq \sigma(x_j) = \sigma(s_j x_1) = \sigma(y_j)$. Secondly, observe 
      \begin{equation}
        s_i^* s_i y_1 = s_i^* s_i x_1 e = s_i^* s_i x_1 x_1^* s_i^* s_i x_1 \, e = x_1 x_1^* s_i^* s_i x_1 e = x_1 e = y_1 \nonumber
      \end{equation}
      and hence $y_1 \in D_{s^*_is_i}$ for every $i = 2, \dots, m+1$, proving that $y_1, \dots, y_{m+1}$ are all $\mathcal{L}$-related (cf., Lemma~\ref{lemma_action_lclasses}). Finally, the points $y_1$ and $y_i$ are connected by a path of length less than $R$ by construction, since $s_i y_1 = y_i$, proving the claim.
    \end{proof}

    \begin{remark}
      The difficulty of the preceding proof is the fact that the elements $x_1, \dots, x_{m+1} \in S$ need not be $\mathcal{L}$-related, that is, they might sit in different Sch\"{u}tzenberger classes of $S$. Moreover, the edges connecting $\sigma(x_1), \dots, \sigma(x_{m+1})$ in $G(S)$ need not be present in a certain $\mathcal{L}$-class, and thus we have to move the point $x_1 \in S$ via multiplication with a suitable projection $e$ in order to replicate those edges in a certain $\mathcal{L}$-class $L_e \subset S$. Therefore, what the proof above actually says is that the local structure of $G(S)$, i.e., a certain $R$-ball $B_R(\sigma(x_1)) \subset G(S)$, may be seen in a $\mathcal{L}$-class $L_e \subset S$ provided that $e \in E(S)$ is sufficiently small. In particular, the left Cayley graph of $G(S)$ naturally is the inductive limit of the Sch\"{u}tzenberger graphs of $S$ (see also the proof of Proposition~\ref{prop_sa_gsa}).
    \end{remark}
    
    Recall that not every $K$-labeled graph of bounded geometry is the left Cayley graph of a group generated by the set $K$ (e.g., the Petersen graph). The following proposition shows that a large class of graphs can be realized as Sch\"utzenberger graphs. We would like to thank N\'{o}ra Szak\'{a}cs for pointing out the proof of the following proposition.
    \begin{proposition}\label{prop:lclasses}
      Let $\mathcal{G} = (V, E)$ be a non-empty, connected and undirected graph without multiple edges and suppose that every vertex is connected with itself via a loop. Then there is an inverse semigroup $S$ and an $\mathcal{L}$-class $L \subset S$ such that $\Lambda_L$ and $\mathcal{G}$ are isomorphic as graphs.
    \end{proposition}
    \begin{proof}
      We will only sketch the main ideas of the proof (see also~\cite{S90,MS15}). It is useful for the construction to think of the undirected edges of $\mathcal{G}$ 
      as a pair of edges with opposite orientations. We will denote these as $E \sqcup E^*$, where $(v_2, v_1)^* = (v_1, v_2)$. Fix an arbitrary vertex $v_0 \in V$ and consider the set of cycles in $\mathcal{G}$ starting at $v_0$:
      \begin{equation}
        \mathcal{C}\left(v_0\right) \, := \, \left\{ \, (v_0, v_p) \dots (v_2, v_1 ) (v_1, v_0) \;\;\; \text{such that} \;\; (v_{i + 1}, v_i), (v_0, v_p) \in E \sqcup E^* \,\right\}. \nonumber
      \end{equation}
      Consider the inverse semigroup $S$ formally generated by $V \sqcup E \sqcup E^*\sqcup \{0\}$ and with relations given by:
      \begin{itemize}
        \item $v = v^2 = v^* = (v, v)$ for all $v \in V$.
        \item $v_2 (v_2, v_1) = (v_2, v_1) = (v_2, v_1) v_1$ for all edges $(v_2, v_1) \in E$.
        \item $v_3 (v_2, v_1) = 0$ if $v_2 \neq v_3$.
        \item $(v_2, v_1) v_3 = 0$ if $v_1 \neq v_3$.
        \item $\omega = v_0^*v_0$ for all $\omega \in \mathcal{C}(v_0)$.
      \end{itemize}
      Then, taking $K := V \sqcup E  \sqcup E^*\sqcup \{0\}$, the left Sch\"{u}tzenberger graph of the $\mathcal{L}$-class of $v_0$ can naturally be seen as oriented paths in $\mathcal{G}$ starting at $v_0$. Indeed, note that non-zero elements in $S$ are formal expressions $\mathfrak{p} = (v_p, v_{p-1}) \dots (v_2, v_1)$, where $(v_i, v_{i-1})$ are edges in $\mathcal{G}$, and $\mathfrak{p} \mathcal{L} v_0$ if and only if $\mathfrak{p}$ starts at $v_0$. Moreover, if $\mathfrak{p}, \mathfrak{q}$ are two paths in $\mathcal{G}$ with the same initial vertex $v$ and final vertex $w$, it follows that $\mathfrak{q}^*\mathfrak{p} = v$ and $\mathfrak{p}\mathfrak{q}^* = w$. Therefore
      \begin{equation}
        \mathfrak{p} = \mathfrak{p} v = \mathfrak{p} \mathfrak{q}^* \mathfrak{p} \;\; \text{and} \;\; \mathfrak{q}^* = \mathfrak{q}^* w = \mathfrak{q}^* \mathfrak{p} \mathfrak{q}^*, \nonumber
      \end{equation}
      which implies that $\mathfrak{p} = \mathfrak{q}$ whenever they share the initial and final vertices. Hence the map $\mathfrak{p} \mapsto r(\mathfrak{p})$, sending each path $\mathfrak{p}$ to its final vertex, is a natural bijection between the $\mathcal{L}$-class of $v_0$ in $S$ and the graph $\mathcal{G}$. Observe, as well, that two elements $\mathfrak{p}, \mathfrak{q} \in S$ are joined by an edge in $\Lambda_S$ if one is a prefix of the other, and therefore the map above is a graph isomorphism.
    \end{proof}
    \begin{remark}
      The construction of $S$ in the proof of the preceding proposition can be also done in the more general case where the graph $\mathcal{G} = (V,E)$ is $K$-labeled with $K$  a set, as long as the labeling of the graph is \textit{deterministic} (see~\cite{MS15}). Moreover, in that case the graph isomorphism respects the $K$-labeling. Also, observe that the construction above is not the so-called \textit{path inverse semigroup} (see~\cite{AH75}).

      Note, as well, that we require the vertices of $\mathcal{G}$ to be decorated with a loop. This condition is irrelevant from a large-scale geometry point of view and, therefore, any simple connected graph can be quasi-isometrically realized as a Sch\"utzenberger graph of an inverse semigroup.
    \end{remark}

    In the final part of the section we will state some results concerning the graph $\Lambda_S$ when seen as a metric space with the path length metric. The next useful lemma relates different distances one may consider in inverse semigroups. In particular, if $S = \langle K \rangle$ we denote by $\ell(\cdot)$ the minimal length of a word in the alphabet $K$, by $d_S$ the path distance in $\Lambda(S, K)$ and by $d_{G(S)}$ the path distance in the left Cayley graph of $G(S)$ with respect to $\sigma(K)$.
    \begin{lemma}\label{lemma_dist_lclasses}
      Let $S = \langle K \rangle$ be an inverse semigroup, and let $\sigma \colon S \rightarrow G(S)$ be the canonical projection onto the maximal homomorphic image.
      \begin{enumerate}
        \item \label{dist_lclasses_1} For any $s, x \in S$ such that $x \in D_{s^*s}$ we have
          \begin{equation}
            d_{G\left(S\right)}\left(\sigma\left(x\right), \sigma\left(sx\right)\right) \leq d_{S}\left(x, sx\right) \leq d_S \left(s^*s, s\right) \leq \ell(s). \nonumber
          \end{equation}
        \item \label{dist_lclasses_2} For any $s, x \in S$ such that $xx^* = s^*s$ (hence, in particular, $x \in D_{s^*s}$) we have
          \begin{equation}
            d_S\left(x, sx\right) = d_S\left(s^*s, s\right). \nonumber
          \end{equation}
      \end{enumerate}
    \end{lemma}
    \begin{proof}
      For (\ref{dist_lclasses_1}), observe that if $\ell(s) = d$ and $s = k_d \dots k_1 \in K^d$, then $s^*s$ and $s$ are joined by a path in the $\mathcal{L}$-class $L_{s^*s}$ labeled by $k_d, \dots, k_1$, and thus $\ell(s) \geq d_S(s^*s, s)$. Moreover, for $x\in D_{s^*s}$ a geodesic path joining $s^*s$ with $s$ will define by multiplication from the right with $x$ a path of the same length joining $x$ with $sx$ on $L_{x^*x}$ and, therefore, $d_{S}(x, sx) \leq d_S(s^*s, s)$. Similarly, any geodesic path joining $x$ with $sx$ defines via the quotient map $\sigma$ a path joining $\sigma(x)$ with $\sigma(sx)$ in the Cayley graph of $G(S)$, proving the last inequality.

      Part (\ref{dist_lclasses_2}) follows from (\ref{dist_lclasses_1}) since
      \begin{equation}
        d_S\left(x, sx\right) \leq d_S\left(s^*s, s\right) = d_S\left(xx^*, sxx^*\right) \leq d_S\left(x, sx\right), \nonumber
      \end{equation}
      where the last inequality is, again, due to the fact that a geodesic connecting $x$ and $sx$ in $\Lambda_S$ defines a path between $xx^*$ and $sxx^*$ when multiplied on the right by $x^*$.
    \end{proof}

    Note that Lemma~\ref{lemma_dist_lclasses}~(\ref{dist_lclasses_2}) relates distances between pairs of points in different $\mathcal{L}$-classes, since $x$ and $xx^*$ need not be $\mathcal{L}$-related in general. Moreover, Lemma~\ref{lemma_dist_lclasses} precisely defines the notion of \textit{right-invariance} of a metric in the context of inverse semigroups.

    The following proposition, which might be known to experts, proves that some of the Sch\"{u}tzenberger graphs $\Lambda_L$ are isomorphic to each other.
    \begin{proposition}\label{prop_dclasses}
      Let $S$ be a countable inverse semigroup. Given $s \in S$, let $\Lambda_{s^*s}$ and $\Lambda_{ss^*}$ be the Sch\"{u}tzenberger graphs of the $\mathcal{L}$-classes of $s^*s$ and $ss^*$, respectively. Then
      \begin{equation}
        \rho \colon \Lambda_{s^*s} \rightarrow \Lambda_{ss^*}, \;\;\; \text{where} \;\; \rho\left(x\right) := xs^* \nonumber
      \end{equation}
      defines a graph isomorphism between $\Lambda_{s^*s}$ and $\Lambda_{ss^*}$.
    \end{proposition}
    \begin{proof}
      First note that $(xs^*)^* xs^* = sx^*xs^* = ss^*ss^* = ss^*$ and thus $xs^* \mathcal{L} ss^*$. If $xs^* = ys^*$ then $x = xx^*x = xs^*s = ys^*s = yy^*y = y$ for any $x, y \in \Lambda_{s^*s}$, which proves that $\rho$ is injective. Given an arbitrary $t \in \Lambda_{ss^*}$, observe that $t^*t = ss^*$ and, therefore, $t = tss^* = \rho(ts)$ since $ts \mathcal{L} s^*s$, proving that $\rho$ is a bijection. Finally, since $\rho$ is defined by right multiplication it preserves adjacency. Indeed, the points $x, kx \in \Lambda_{s^*s}$ are joined by an edge labeled by $k$ if and only if the points $xs^*, kxs^* \in \Lambda_{ss^*}$ are joined by an edge labeled by $k$.
    \end{proof}

    We conclude this section introducing quasi-isometries, which are an important special case of coarse equivalences between metric spaces. This concept is central when viewing an infinite discrete group as a (coarse) geometric object. Given two extended metric spaces $(X, d_X)$ and $(Y, d_Y)$, we say that they are \textit{quasi-isometric} (see~\cite{R03,NY12,GK13}) if there is a map $\phi \colon X \rightarrow Y$ such that
    \begin{enumerate}
       \item \label{def_qi_emb} There are some constants $M>0$, $C \geq 0$ such that, for any $x, x' \in X$
        \begin{equation}
          \frac{1}{M} d_X \left(x, x'\right) - C \leq d_Y\left(\phi\left(x\right), \phi\left(x'\right)\right) \leq M d_X\left(x, x'\right) + C. \nonumber
        \end{equation}
      \item \label{def_qi_surj} There exists $R > 0$ such that for any $y \in Y$ there is $x \in X$ with $d_Y\left(y, \phi\left(x\right)\right) \leq R$.
    \end{enumerate}
    A function $\phi$ satisfying both conditions above is called a \textit{quasi-isometry}, while an injective map $\phi$ satisfying only (\ref{def_qi_emb}) is called a \textit{quasi-isometric embedding}.
     
    In the following proposition we generalize a well known result in the case of groups (see~\cite[Theorem~1.3.12]{NY12}). See also \cite[Proposition~4]{GK13} for a 
    similar statement for semigroups considered as semi-metric spaces.
    \begin{proposition}\label{prop_uniq}
       Let $S$ be an inverse semigroup and let $K_1, K_2 \subset S$ be two finite and symmetric generating sets. Then the graphs $\Lambda(S, K_1)$ and $\Lambda(S, K_2)$ are quasi-isometric.
    \end{proposition}
    \begin{proof}
      The identity function $id \colon \Lambda(S, K_1) \rightarrow \Lambda(S, K_2)$ is a quasi-isometry. Indeed, since it is surjective it is enough to prove there is some constant $M > 0$ such that
      \begin{equation}
        \frac{1}{M} d_{K_1}\left(x, y\right) \leq d_{K_2}\left(x, y\right) \leq M d_{K_1}\left(x, y\right)\;\; \text{for any} \;\, x,y\in S\;, \nonumber
      \end{equation}
      where $d_{K_i}$ denotes the path distance in the graph $\Lambda(S, K_i)$. Note that if $x$ and $y$ are not $\mathcal{L}$-related then $d_{K_1}(x, y) = d_{K_2}(x, y) = \infty$, so we may suppose that $x, y$ belong to the same $\mathcal{L}$-class. In this case the inequalities follow by choosing
      \begin{equation}
        M := \max \left\{ d_{K_1}\left(k^*k, k\right)\,,\, d_{K_2}\left(k^*k, k\right)\; \mid k \in K_1 \cup K_2 \right\}\,. \nonumber
      \end{equation}
    \end{proof}

    %----------------------------------------------------------------------------------------
    % BOUNDED PROPAGATION AND RS
    %----------------------------------------------------------------------------------------
    \subsection{The uniform Roe-algebra}\label{sec_bnbprop_thm}
    In this section we will show that the graph $\Lambda_S:=\sqcup_e L_e$ (seen as a metric space with the path length metric) allows to witness the algebra $\mathcal{R}_S$ as the uniform Roe algebra $C_u^*(\Lambda_S)$. We begin introducing an important notion for the semigroup $S$ (see in relation with Theorem~\ref{thm_mtrspc}).
    \begin{definition}\label{def_fl}
      Let $S = \langle K \rangle$ be an inverse semigroup with countable and symmetric generating set $K$ and let $L \subset S$ be an $\mathcal{L}$-class.
      \begin{enumerate}
        \item We say the pair $(L, K)$ admits a \textit{finite labeling}, or is \textit{FL} for short, if there are $C \geq 1$ and a finite $K_1 \subset K$ such that for any $x, y \in L$ with $y \in K x$, we have that $y \in K_1^px$ for some $p \in \{1, \dots, C\}$, where $K_1^p$ denotes the words of length $p$ in the alphabet $K_1$.
        \item We say the pair $(S, K)$ admits a \textit{finite labeling} if every $\mathcal{L}$-class of $S$ is FL uniformly over the classes, that is, if there are $C \geq 1$ and a finite $K_1 \subset K$ such that, for any $x, y \in S$ with $x \mathcal{L} y$ and $y \in K x$, we have that $y \in K_1^p x$ for some $p \in\{ 1, \dots, C\}$.
      \end{enumerate}
    \end{definition}
    \begin{remark}
      We say \textit{the pair $(S, K)$ has condition FL} instead of \textit{$S$ has condition FL} since admitting a finite labeling may depend on the generating set $K$. However, as most times $K$ is fixed, we may say \textit{$S$ admits a finite labeling}, meaning the pair $(S, K)$ admits a finite labeling.
    \end{remark}
    The preceding definition, although technical, is essential to various arguments in the article. It is, in particular, an algebraic characterization of the equality of C*-algebras $\mathcal{R}_S = C_u^*(\Lambda_S)$ (see Theorem~\ref{thm_mtrspc}).

    The following proposition gives sufficient conditions for an inverse semigroup to be FL (see Proposition~\ref{prop_ex_fl}), showing there are important classes of examples admitting finite labelings. Recall that an inverse semigroup is called \textit{F-inverse} if each $\sigma$-class has exactly one greatest element. All free inverse semigroups are F-inverse and, moreover, every inverse semigroup has an F-inverse cover (see~\cite[p.~230]{L98}).
    \begin{proposition}\label{prop_ex_fl}
      The following classes of inverse semigroups admit finite labelings:
      \begin{enumerate}
        \item \label{class_fl_fg} The class of finitely generated inverse semigroups.
        \item \label{class_fl_finv} The class of F-inverse semigroups such that $\Lambda_S$ is of bounded geometry.
      \end{enumerate}
    \end{proposition}
    \begin{proof}
      First we show that every finitely generated inverse semigroup belongs to the FL class. In fact, take $C := 1$ and $K_1 := K$, which is finite by assumption. For the second statement let $S = \langle K \rangle$ be an F-inverse semigroup such that $\Lambda_S$ is of bounded geometry. Denote the maximal group homomorphic image $G(S)$ simply by $G$. Consider then the projection $\sigma$ restricted to the following set of greatest elements:
      \[
        A_{\mathrm{great}}:= \left\{s \in S \mid s \; \text{is greatest in}~\sigma(s) \;\text{and}\; \exists x \in D_{s^*s} \, \text{with} \; d_S\left(x, sx\right) \leq 1 \, \right\}\;,
      \]
      i.e., consider
      \begin{equation}
        \sigma \colon\,A_{\mathrm{great}} \rightarrow \, \left\{\, g \in G \; \mid \; d_G\left(1_{G}, g\right) \leq 1 \, \right\}. \nonumber
      \end{equation}
      Note that $A_{\mathrm{great}}$ is not empty because $1 \in A_{\mathrm{great}}$ and since the greatest element in each class is unique the preceding map is injective. Moreover, the right hand side is a finite set since $\Lambda_S$ is of bounded geometry (see Proposition~\ref{prop_bddgeom}) and, hence, $A_{\mathrm{great}}$ is finite too. Therefore there is a finite $K_1 \subset K$ such that every element in $A_{\mathrm{great}}$ is a word in $K_1$ of length less than $C:=\max\{\ell(s)\mid s\in A_{\mathrm{great}}\}$. It then follows that $S$ is FL for those $K_1$ and $C$. In fact, let $x\mathcal{L}y$ with $y=kx$ for some $k\in K$ (i.e., $d_S(x,y)\leq 1$). Denote by $s_0$ the greatest element in $\sigma(k)$. Since $D_{k^*k}\subset D_{s_0^*s_0}$ we have $sx=kx=y$ and, therefore, $y\in K_1^px$ for some $p\in \{1,\dots, C\}$, proving that $S$ is FL.
    \end{proof}

    The following proposition gives necessary conditions for an inverse semigroup to admit a finite labeling. Even though its proof is straightforward it highlights some key ideas behind the definition. Note that the construction of $\Lambda(S,K)$, carried out in Subsection~\ref{sec_dist} for a generating set $K$, can be done similarly for any subset of labels $K_1 \subset K$, and the corresponding graph will be denoted by $\Lambda(S,K_1)$, which is then a subgraph of $\Lambda(S,K)$.
    \begin{proposition}\label{prop_fl_vs_qi}
      Let $S = \langle K \rangle$ be an FL inverse semigroup, and let $K_1 \subset K$ be as in Definition~\ref{def_fl}. Then the following hold:
      \begin{enumerate}
        \item \label{fl_implies_qi} The identity map on vertices $\Lambda(S, K) \rightarrow \Lambda(S, K_1)$ is a quasi-isometry.
        \item \label{fl_implies_fg_mod_e} Every $s \in S$ can be written as a word $s = k_d \dots k_1 e$, where $e \in E(S)$ and $k_i \in K_1$, that is, $S$ is \textit{finitely generated modulo the idempotents}.
      \end{enumerate}
    \end{proposition}
    \begin{proof}
      (\ref{fl_implies_qi}) follows from the fact that for every $x, y \in S$
      \begin{equation}
        d_K\left(x, y\right) \leq d_{K_1}\left(x, y\right) \leq C d_K \left(x, y\right) \nonumber
      \end{equation}
      where $d_K$ and $d_{K_1}$ denote the path distances in $\Lambda(S, K)$ and $\Lambda(S, K_1)$, respectively, and $C \geq 1$ is as in Definition~\ref{def_fl}.

      (\ref{fl_implies_fg_mod_e}) holds since, by the hypothesis, any $s \in S$ can be written as $s = k_d \dots k_1$ where $k_i \in K_1$ that is, $s = ss^*s = k_d \dots k_1 s^*s$.
    \end{proof}

    Lastly, the following gives an alternative characterization of admitting a finite labeling.
    \begin{proposition} \label{prop:fl:newchr}
      Let $S = \langle K \rangle$ be an inverse semigroup. The pair $(S, K)$ admits a finite labeling if, and only if, for every $R \geq 0$ there is a finite $F \subset S$ such that if $d(s^*s, s) \leq R$ then there is some $m \in F$ such that $m s^*s = s$ (or, equivalently, $m \geq s$).
    \end{proposition}
    \begin{proof}
      Let $K_1 \subset K$ and $C \geq 1$ witness the FL condition of $(S, K)$ as in Definition~\ref{def_fl}. Given a radius $R > 0$ the set $F := \cup_{p = 1}^{RC} K_1^p$ is finite. Moreover, it labels all paths in $S$ in the alphabet $K$ and of length at most $R$. That is, whenever $d(s^*s, s) \leq R$ there is some $m \in F$ such that $ms^*s = s$, as required.

      For the converse, let $R = 1$ and let $F$ be as in the statement. As it is finite, there is a finite $K_1 \subset K$ such that every element $m \in F$ is a word in $K_1$. Putting $C$ as the maximum length of the words in $F$, it then follows that $K_1$ and $C$ witness the FL condition of $(S, K)$.
    \end{proof}
    \begin{remark} \label{rem:fl:cover}
      The preceding proposition gives a geometric meaning to condition FL. Indeed, let $C_R$ be the set of points of $S$ such that $d(s^*s, s) \leq R$. We might see $C_R$ as a \textit{cylinder}, in which case admitting a finite labeling asserts that $C_R$ has a finite \textit{cover}, meaning every point $s \in C_R$ is below some point in $F$, where $F$ is a finite set in $S$ as in the preceding proposition (see also Figure~\ref{fig:fl:cover} below).
      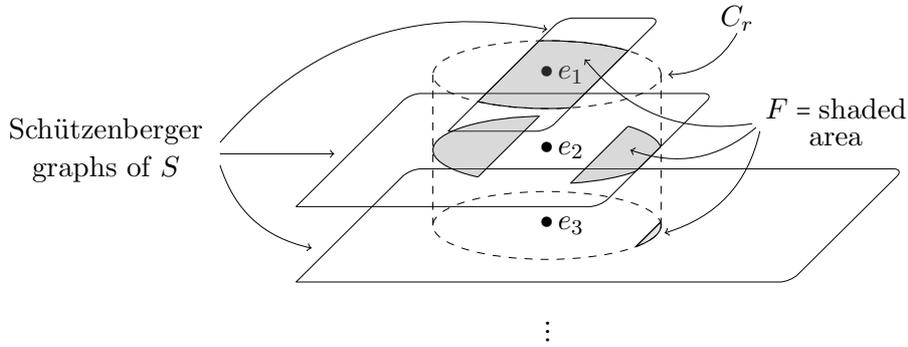
\begin{figure}[ht]
        \centering
        \begin{tikzpicture}
          \node at (-3, 0) {\text{Sch\"{u}tzenberger}};
          \node at (-3, -0.5) {\text{graphs of $S$}};
          % L-classes
          \draw [very thin, rounded corners] (1.5, 0) -- (2.8, 0) -- (4.3, 1.5) -- (3, 1.5) -- (1.5, 0);
          \draw [very thin, rounded corners] (-0.5, -1) -- (3.5, - 1) -- (5, 1.5 - 1) -- (1, 1.5 - 1) -- (-0.5, 0 - 1);
          \draw [very thin, rounded corners] (-0.5, -2) -- (-0.5 + 6.5, - 2) -- (-0.5+6.5+1.5, 1.5 - 2) -- (-0.5 + 1.5, 1.5 - 2) -- (-0.5, 0 - 2);
          % Projections and dots below the L-classes
          \node at (3, 0.75) {$\bullet \, e_1$};
          \node at (3, 0.75 - 1) {$\bullet \, e_2$};
          \node at (3, 0.75 - 2) {$\bullet \, e_3$};
          \node at (2.81, 0.75 - 3.4) {$\vdots$};
          % Arrows pointing to the L-classes
          \draw [very thin, ->] (-1.5, -0.15) to[bend left] (2.8, 1.35);
          \draw [very thin, ->] (-1.5, -0.30) to (0, -0.3);
          \draw [very thin, ->] (-1.5, -0.45) to[bend right] (-0.3, 0.45 - 2);
          % Now the cylinder C_r
          \draw [very thin, dashed] (2.8, 0.75) ellipse (15mm and 4.5mm);
          \draw [very thin, dashed] (2.8, 0.75 - 2) ellipse (15mm and 4.5mm);
          \draw [very thin, dashed] (2.8 - 1.5, 0.75) -- (2.8 - 1.5, 0.75 - 2);
          \draw [very thin, dashed] (2.8 + 1.5, 0.75) -- (2.8 + 1.5, 0.75 - 2);
          \node at (5.3, 1.5) {$C_r$};
          \draw [very thin, ->] (5.3, 1.3) to[bend left] (2.8 + 1.6, 0.75); 
          % Part of F_r in the highest level
          \begin{scope}
            \clip (1.5, 0) -- (2.8, 0) -- (4.3, 1.5) -- (3, 1.5) -- (1.5, 0);
            \draw [fill = gray, fill opacity = 0.3] (2.8, 0.75) ellipse (15mm and 4.5mm);
          \end{scope}
          \begin{scope}
            \clip (2.8, 0.75) ellipse (15mm and 4.5mm);
            \draw (2.8, 0) -- (4.3, 1.5);
            \draw (3, 1.5) -- (1.5, 0);
          \end{scope}
          % Part of F_r in the medium level
          \begin{scope}
            \clip (1.5, -1) -- (3, 0.5) -- (1, 1.5 - 1) -- (-0.5, -1) -- (1.5, -1);
            \draw [fill = gray, fill opacity = 0.3] (2.8, 0.75 - 1) ellipse (15mm and 4.5mm);
          \end{scope}
          \begin{scope}
            \clip (2.8, 0.75 - 1) ellipse (15mm and 4.5mm);
            \draw (1.5, -1) -- (3, 0.5);
          \end{scope}
          \begin{scope}
            \clip (2.8, -1) -- (4.3, 0.5) -- (5, 1.5 - 1) -- (3.5, - 1) -- (2.8, -1);
            \draw [fill = gray, fill opacity = 0.3] (2.8, 0.75 - 1) ellipse (15mm and 4.5mm);
          \end{scope}
          \begin{scope}
            \clip (2.8, 0.75 - 1) ellipse (15mm and 4.5mm);
            \draw (2.8, -1) -- (4.3, 0.5);
            \draw (3.5, - 1) -- (5, 1.5 - 1);
          \end{scope}
          % Part of F_r in the third level
          \begin{scope}
            \clip (2.8, 0.75 - 2) ellipse (15mm and 4.5mm);
            \draw (3.5, -2) -- (5, -0.5);
            \draw [fill = gray, fill opacity = 0.2] (3.5, -2) -- (5, -0.5) -- (7, -0.5) -- (5.5, -2) -- (3.5, -2);
          \end{scope}
          \begin{scope}
            \clip (3.5, -2) -- (5, -0.5) -- (7, -0.5) -- (5.5, -2) -- (3.5, -2);
            \draw (2.8, 0.75 - 2) ellipse (15mm and 4.5mm);
          \end{scope}
          % Say what is F_r
          \node at (6.6, 0.3) {$F = \text{shaded}$};
          \node at (6.6, 0.3-0.4) {$\text{area}$};
          \draw [very thin, ->] (5.5, 0.1) to[bend left] (3.3, 0.95);
          \draw [very thin, ->] (5.55, 0.05) to[bend left] (3.9, 0.75 - 1);
          \draw [very thin, ->] (5.6, 0) to[bend left] (4.35, 0.65 - 2);
        \end{tikzpicture}
        \caption{Rough depiction of the \textit{finite labeling} condition as stated in Proposition~\ref{prop:fl:newchr}. Given a radius $R \geq 0$, the \textit{cylinder} $C_R$ is the set of points of $S$ such that $d\left(s^*s, s\right) \leq R$.}
        \label{fig:fl:cover}
      \end{figure}
    \end{remark}

    \begin{example} \label{ex:fl:grpcase}
      Following Proposition~\ref{prop:fl:newchr} observe that the condition FL is trivially satisfied when $S = G$ is a countable group. Indeed, recall that any countable and discrete group $G$ can be equipped with a unique (up to coarse equivalence) proper and right invariant metric $d$, yielding a metric space $(G, d)$ of bounded geometry (cf., \cite[Proposition~5.5.2]{BO08}). Therefore the $r$-ball $B_r(1) \subset G$ is a finite set and it trivially satisfies condition FL as stated in Proposition~\ref{prop:fl:newchr}. Moreover, observe such a metric on an inverse semigroup needs not be unique up to coarse equivalence, and, hence, the dependence of the FL condition on the generating set $K$ reflects this fact.  
    \end{example}

    \begin{example}\label{ex_fl}
      The semigroup $S = (\mathbb{N}, \min)$, where $n \cdot m := \min\{n, m\}$, does not admit a finite labeling. Indeed, the generating set must necessarily be $K = S = E(S)$, and hence $\Lambda_S=\sqcup_{n\in\mathbb{N}}\{n\}$, i.e., $\Lambda_S$ is formed by countably many isolated points that are pairwise at infinite distance, where any vertex $\{n\}$ has infinitely loops labeled by $m\geq n$. Moreover, for any finite $K_1 \subset K$ there is an $n \in \mathbb{N}$ such that $n \geq k$ for any $k \in K_1$. Letting $x = y := n$ in Definition~\ref{def_fl} it then follows that $S$ is not FL, even though any $\mathcal{L}$-class consisting of a single point is trivially FL. Similarly, the graph $\Lambda_T$ associated to the semigroup $T := (\mathbb{N}, \max)$ consists of infinitely many isolated points pairwise at infinite distance, and any vertex $\{n\}$ has precisely $n$ loops labeled by $k\in\{1,\dots,n\}$. In this case, $\Lambda_T$ is FL because $T$ is F-inverse. Furthermore, observe that neither $S$ nor $T$ are finitely generated.
    \end{example}

    We turn to the proof of the main result, i.e., we want to compare the algebras $\mathcal{R}_S$ and $C_u^*(\Lambda_S)$. The following preliminary result shows that one inclusion always holds. Note that, in general, the uniform Roe algebra decomposes as a direct sum over the algebras associated to the corresponding Sch\"utzenberger classes
    \[
     C^*_u(\Lambda_S)\cong \prod_{e\in E(S)} C^*_u(\Lambda_{L_e})\;.
    \]
    Moreover, for any $\mathcal{L}$-class $L$, the orthogonal projections $p_L$ onto the subspace $\ell^2(L)$ are central.
    
    \begin{proposition}\label{prop_easy_containment}
      Let $S$ be an inverse semigroup. Let $\Lambda_S$ be the disjoint union of the left-Sch\"{u}tzenberger graphs of $S$. Then $\mathcal{R}_S \subset C_u^*(\Lambda_S)$.
    \end{proposition}
    \begin{proof}
      Note first that any $f \in \ell^{\infty}(S)$ corresponds to a diagonal operator and hence has propagation $0$. Moreover, the propagation of the generators $V_s$, where $s \in S$, is given by 
      \[
       p(V_s):=\sup_{x\in D_{s^*s}} d(x,sx) = d(s^*s,s)\,.
      \]
      In fact, since $s^*s\in D_{s^*s}$ we have that $p(V_s)\geq d(s^*s,s)$. The reverse inequality follows from Lemma~\ref{lemma_dist_lclasses}.
    \end{proof}
    \begin{theorem}\label{thm_mtrspc}
      Let $S = \langle K \rangle$ be a countable and discrete inverse semigroup such that $\Lambda_S$ is of bounded geometry. Then the following conditions are equivalent:
      \begin{enumerate}
        \item \label{thm_fl} The pair $(S, K)$ admits a finite labeling (see Definition~\ref{def_fl}).
        \item \label{thm_calg} The $*$-algebras $\mathcal{R}_{S,\mathrm{alg}}$ and $C^*_{u,\mathrm{alg}}(\Lambda_S)$ are equal, and hence
        \[
         \mathcal{R}_S =C^*_u(\Lambda_S)\;.
        \]
      \end{enumerate}
    \end{theorem}
    \begin{proof}
      (\ref{thm_fl}) $\Rightarrow$ (\ref{thm_calg}). Given an operator $T \in C^*_{u, alg}(\Lambda_S)$, say of propagation $R > 0$, let $F \subset S$ be as in Proposition~\ref{prop:fl:newchr}. Observe that for every pair $x, y \in S$ such that $d(x, y) \leq R$ there is $t \in F$ such that $tx = y$. Let $t_{x, y} \in F$ be such a possible choice and consider the functions $\xi_s \in \ell^{\infty}(S)$, where $s\in S$, defined by
      \begin{equation}
        \xi_s\left(y\right) :=
          \left\{ \begin{array}{lc}
            \langle \delta_y, T \delta_{s^*y} \rangle & \, \text{if} \;\; y \mathcal{L} s^*y \;\; \text{and} \;\; s = t_{s^*y, y}, \\[2mm]
            0 & \text{otherwise}.
          \end{array} \right. \nonumber
      \end{equation}
      We claim
      \begin{equation}
        T = \sum_{s \in S} \xi_s V_s. \nonumber
      \end{equation}
      First note that the sum above is finite, as $\xi_s = 0$ when $s \not\in F$.  Moreover, on the one hand, if $x$ and $y$ are not $\mathcal{L}$-related then $\langle \delta_y, T \delta_x \rangle = 0$ since $T$ is of bounded propagation. In addition, if $x \in D_{s^*s}$ then $x \mathcal{L} sx$ and, therefore, $sx \neq y$, which implies that $\langle \delta_y, (\sum_{s \in S} \xi_s V_s) \delta_x \rangle = 0$. On the other hand, if $x \mathcal{L} y$ then:
      \begin{align}
        \left\langle \delta_y, \left(\sum_{s \in S} \xi_s V_s\right) \delta_x \right\rangle 
                   = \left\langle \delta_y, \sum_{\substack{s \in F \\ x \in D_{s^*s}}} \xi_s\left(sx\right) \delta_{sx} \right\rangle 
                   = \sum_{\substack{s \in F \\ sx = y}} \xi_s\left(y\right) 
                   = \langle \delta_y, T \delta_x \rangle, \nonumber
      \end{align}
      since, by construction, there is exactly one $s \in F$ such that $sx = y$ and $\xi_s(y) \neq 0$, namely $s = t_{x, y}$.

      (\ref{thm_calg}) $\Rightarrow$ (\ref{thm_fl}). Suppose $S$ is not FL. Thus, by Proposition~\ref{prop_ex_fl}, $S$ is not finitely generated. Let $K = \left\{k_1, k_2, \dots\right\}$ be a numeration of $K$, and let $K_n := \left\{k_1, \dots, k_n\right\}$ be the first $n$ generators. Note that, since $S$ is not FL, for every $n \in \mathbb{N}$ there are points $x_n, y_n \in S$ such that $x_n \mathcal{L} y_n$, $y_n \in K x_n$ and $x_n$ and $y_n$ are not joined by a path labeled by $K_n$ of length less than $n$. Note, as well, that since $\Lambda_S$ is locally finite (as it is of bounded geometry) we may suppose that $x_n \neq x_{n'}$ for any $n \neq n'$. Consider the operator:
      \begin{equation}
        T \colon \ell^2\left(S\right) \rightarrow \ell^2\left(S\right), \quad T \delta_z :=
          \left\{
            \begin{array}{lc}
              \delta_{y_n} & \; \text{if} \; \; z = x_n, \nonumber \\[2mm]
              0 & \text{otherwise}.
            \end{array}
          \right. \nonumber
      \end{equation}
      The operator $T$ has finite propagation since $\sup_{n \in \mathbb{N}} d(x_n, y_n) \leq 1$, and, thus, $T \in C^*_{u, \mathrm{alg}}(\Lambda_S)$. Moreover, we will show that $T$ cannot be approximated by elements in $\mathcal{R}_{S, \mathrm{alg}}$ and, therefore, $T \not\in \mathcal{R}_S$. Indeed, given any $\sum_{i = 1}^m f_i V_{s_i} \in \mathcal{R}_{S, \mathrm{alg}}$, let $n \in \mathbb{N}$ be sufficiently large so that $s_i$ is a word in $K_n$ for every $i = 1, \dots, m$ and $n$ is greater than the maximum length of the elements $s_i$. In this case
      \begin{equation}
        \left|\left| T - \sum_{i = 1}^m f_i V_{s_i}\right|\right| \geq \left|\left| T\left(\delta_{x_n}\right) - \left(\sum_{i = 1}^m f_i V_{s_i}\right) \, \left(\delta_{x_n}\right)\right|\right|_{\ell^2\left(S\right)} = \left|\left| \delta_{y_n} - \sum_{\substack{i = 1 \\ x_n \in D_{s_i^*s_i}}}^m f_i\left(s_i x_n\right) \delta_{s_i x_n} \right|\right|_{\ell^2\left(S\right)} \geq 1, \nonumber
      \end{equation}
      where the last inequality follows from the fact that, by construction, $s_i x_n \neq y_n$ for all $i = 1, \dots, m$.
    \end{proof}
   
    Based on the strategy of the proof of the preceding theorem one can prove a similar result relating the uniform Roe algebra of an $\mathcal{L}$-class in $S$ and the corresponding corner of $\mathcal{R}_S$.
    \begin{theorem}\label{thm_mtrspc_lclasses}
      Let $S = \langle K \rangle$ be a countable and discrete inverse semigroup and let $L \subset S$ be an $\mathcal{L}$-class such that $\Lambda_L$ is of bounded geometry. Denote by $p_L$ be the orthogonal projection from $\ell^2(S)$ onto $\ell^2(L)$. Then the following conditions are equivalent:
      \begin{enumerate}
        \item \label{thm_mtrspc_lclasses_fl} The pair $(L, K)$ admits a finite labeling (see Definition~\ref{def_fl}).
        \item \label{thm_mtrspc_lclasses_calg} The C*-algebras $p_L \mathcal{R}_S p_L$ and $C^*_u(\Lambda_L)$ are equal.
      \end{enumerate}
    \end{theorem}

    We apply next Theorem~\ref{thm_mtrspc} to the classes of inverse semigroups satisfying condition FL (cf., Proposition~\ref{prop_ex_fl}).
    \begin{corollary}\label{cor_finv}
      Let $S$ be a finitely generated inverse semigroup or F-inverse semigroup such that $\Lambda_S$ is of bounded geometry. Then $\mathcal{R}_S = C^*_u(\Lambda_S) = \ell^{\infty}(S) \rtimes_r S$.
    \end{corollary}
    \begin{example}\label{ex_c0}
      Recall from Example~\ref{ex_fl} that the semigroup $S = (\mathbb{N}, \min)$ does not admit a finite labeling. Thus, by Theorem~\ref{thm_mtrspc}, we have $\mathcal{R}_S \neq C^*_u(\Lambda_S)$. Indeed, observe that $\Lambda_S=\sqcup_{n\in\mathbb{N}}\{n\}$  and, therefore, $C^*_u(\Lambda_S) = \ell^{\infty}(\mathbb{N})$. On the other hand, 
      $\mathcal{R}_{S,\mathrm{alg}}=c_{\mathrm{fin}}(\mathbb{N})$ (sequences with finite support) and, hence, 
      $\mathcal{R}_S = c_0(\mathbb{N})$.
    \end{example}

  %----------------------------------------------------------------------------------------
  % QUASI-ISOMETRIC INVARIANTS
  %----------------------------------------------------------------------------------------
  \section{Quasi-isometric invariants}\label{sec_qiinv}
  In this final section we study some of the \textit{large scale properties} of the graph $\Lambda_S$ constructed in Section~\ref{sec_dist} (see also~\cite{GK13} and references therein) and relate these with C*-properties of the reduced semigroup C*-algebra $C^*_r(S)$ and the uniform Roe algebra $C^*_u(\Lambda_S)$. Recall that a property $\mathcal{P}$ is said to be a \textit{quasi-isometric} invariant if, given two quasi-isometric extended metric spaces $(X, d_X)$ and $(Y, d_Y)$ such that $X$ has $\mathcal{P}$, then so does $Y$. We will be particularly interested in two notions, namely \textit{amenability} and \textit{property A}.
  
  According to Day's original definition in \cite{D57}, a semigroup $S$ is amenable if there exists a probability measure on $S$ which is invariant under taking preimages
  (see also~\cite{DN78,GK17}). In the context of inverse semigroups this condition splits into two conditions called {\em domain measurability} and {\em localization} (see \cite[Section~4.1]{ALM19} for details). The former condition captures the dynamical invariance of the measure: $S$ is {\em domain measurable} if there exists a finitely additive (total) probability measure $\mu\colon \mathcal{P}(S)\to [0,1]$ satisfying
  \begin{equation}\label{eq:DM}
   \mu(sA)=\mu(A)\quadtext{for all} s\in S\; \text{and} \; A \subset D_{s^*s}\;.
  \end{equation}
  Amenability and domain measurability of inverse semigroups also allow a {\em F\o lner type  characterization} which we will need later in this section (see, e.g., 
  Theorem~\ref{thm_ame_qi}). For proofs of the following result and additional motivation we refer to Theorems~4.23 and 4.27 in \cite{ALM19}.
  \begin{theorem}\label{def_ame}
    Let $S$ be a countable and discrete inverse semigroup.
    \begin{enumerate}
      \item \label{def_ame_dom} $S$ is domain measurable if and only if for every finite $\mathcal{F} \subset S$ and $\varepsilon > 0$ there is some finite and non-empty $F \subset S$ such that $|s (F \cap D_{s^*s}) \setminus F| \leq \varepsilon |F|$ for every $s \in \mathcal{F}$. 
      \item \label{def_ame_ame} $S$ is amenable if and only if for every finite $\mathcal{F} \subset S$ and $\varepsilon > 0$ there is some finite and non-empty
      $F \subset S$ such that $F \subset D_{s^*s}$ and $|sF\setminus F | \leq \varepsilon |F|$ for all $s \in \mathcal{F}$.
    \end{enumerate}
  \end{theorem}
  In the case of groups amenability and domain measurability coincide and both F\o lner type characterizations in (\ref{def_ame_dom}) and (\ref{def_ame_ame}) above are equivalent. However, these two notions are different in general. 
  The stronger notion of amenability requires, in addition, that the F\o lner sets are localized {\em within} the corresponding domains.
  For example, $S := \mathbb{F}_2 \sqcup \{1\}$, where $1 \in S$ denotes a unit, is a domain measurable semigroup (take $F := \{1\}$), but it is not amenable (since the free group $\mathbb{F}_2$ is not amenable). Lastly, recall that amenability is a quasi-isometric invariant for groups (see~\cite[Theorem~3.1.5]{NY12}).
  \begin{theorem}\label{ame_qiinv_gr}
    Let $G$ and $H$ be quasi-isometric finitely generated groups. If $H$ is amenable then so is $G$.
  \end{theorem}

  \textit{Property A} is a metric property of a space that can be seen as a weak version of amenability (cf., \cite[Example~4.1.2]{NY12}).
  It was introduced by Yu in~\cite{Y00} and has been studied extensively since then (see, for example, \cite{BO08,NY12,R03}).
  \begin{definition}\label{def_a}
    Let $(X, d)$ be a metric space, and let $\{(X_e, d_e)\}_{e \in E}$ be a family of metric spaces.
    \begin{enumerate}
      \item $(X, d)$ has \textit{property A} if for every $\varepsilon > 0$ and $R > 0$ there is $C > 0$ and $\xi \colon X \rightarrow \ell^1(X)$ such that:
      \begin{enumerate}
        \item For every $x \in X$ the function $\xi_x$ is positive, has norm $1$ and has support contained in $B_C(x)$, the ball of center $x$ and radius $C$.
        \item $||\xi_x - \xi_y||_1 \leq \varepsilon$ for every $x, y \in X$ such that $d(x, y) < R$.
      \end{enumerate}
      \item The family $\{(X_e, d_e)\}_{e \in E}$ has \textit{uniform property A} if every $(X_e, d_e)$ satisfies property A with constants $C_e > 0$ which are uniformly bounded, i.e., $C := \sup_{e \in E} C_e < \infty$.
    \end{enumerate}
  \end{definition}
  Property A was introduced as a sufficient condition to ensure that $(X, d)$ coarsely embeds into a Hilbert space (see~\cite{BO08,NY12,R03,Y00} and the survey~\cite{W08}). It is a particularly interesting notion in the context of the so-called \textit{Baum-Connes} conjecture (see~\cite{Y00}). Even though the class of property A groups is plentiful (containing all amenable groups, free groups and being preserved by various constructions, see~\cite[Chapter~4]{NY12}) it turned out that not every group has property A. The only proofs known to the authors of the existence of non-property A groups were initially given by Gromov in~\cite{G03} and then by Osajda in~\cite{O14}, and rely on the coarse embedding of expander graphs into discrete groups (see~\cite{W08}). Even though these constructions are quite exotic, due to the rigid nature of groups, there are several constructions of non-property A metric spaces. For instance, the \textit{coarse disjoint union} of the Cayley graphs of $\{\mathbb{Z}_2^n\}_{n \in \mathbb{N}}$ is a metric space that does not have property A but is not of bounded geometry either (see~\cite{N07} and~\cite[Theorem~4.5.3]{NY12}).
  \begin{example}\label{ex_boxspc_a}
    An interesting class of examples are the so-called \textit{box spaces} (see~\cite[Example~4.4.7]{NY12} and~\cite{R03,Kh12}). We consider the construction in the context of inverse semigroups (see~\cite[Example~4.11]{A19}) as well as \cite{HLS02,W15} for a similar construction for groupoids). Let $G$ be a residually finite group, and let $\{N_i\}_{i \in \mathbb{N}}$ be a descending chain of normal subgroups of finite index of $G$ with trivial intersection. For technical reasons, let $N_\infty := \{1\}$. Then consider $S = \sqcup_{i \in \mathbb{N} \sqcup \{\infty\}} \, G/N_i$, equipped with the operation
    \begin{equation}
      q_i\left(g\right) \cdot q_j\left(h\right) := q_{\min\left\{i, j\right\}} \left(gh\right), \nonumber
    \end{equation}
    where $q_i \colon G \rightarrow G/N_i$ is the canonical quotient map. Observe in particular that the $\mathcal{L}$ and $\mathcal{R}$ relations are equal, i.e., $S$ is a \textit{bundle of groups}. 
  \end{example}  
  This construction is of particular interest because of the following known result (for a proof see~\cite[Theorem~4.4.6]{NY12}):
  \begin{proposition}\label{prop_ex_willett_1}
    Let $G$ be a residually finite group and $S$ be constructed as in Example~\ref{ex_boxspc_a}. Then $\Lambda_S$ has property A if and only if $G$ is amenable.
  \end{proposition}
  In particular, if $G = \mathbb{F}_2$, then every $\mathcal{L}$-class of $S$ has property A, since all of them are either finite or quasi-isometric to a tree, but the graph $\Lambda_S$ does not have property A. Moreover, the C*-algebra $C_r^*(S)$ is non-nuclear and non-exact (by Theorem~\ref{invsem_thm_a} below). Finally, we also mention that the box space construction is also useful to give examples of bounded geometry metric spaces not having property A. The corresponding uniform Roe algebras are examples of non-nuclear C*-algebras having amenable traces (see~\cite[Remark~4.14]{ALLW18R}).

  To finish this brief introduction to property A we recall the following known characterization that motivates our analysis in this section. For a proof see, e.g., \cite[Theorems~5.1.6 and 5.5.7]{BO08}.
  \begin{theorem}\label{thm_gr_a}
    Let $G$ be a finitely generated group. The following are equivalent:
    \begin{enumerate}
      \item The left Cayley graph of $G$ has property A.
      \item The uniform Roe algebra $\ell^{\infty}\left(G\right) \rtimes_r G$ is nuclear.
      \item The reduced group C*-algebra $C^*_r(G)$ is exact.
    \end{enumerate}
  \end{theorem}

    %----------------------------------------------------------------------------------------
    % AMENABILITY AND DOMAIN-MEASURABILITY
    %----------------------------------------------------------------------------------------
    \subsection{Amenability and domain measurability}\label{sec_qi_ame}
    In this subsection we show that domain measurability, defined via the invariance condition in Eq.~\eqref{eq:DM}, is a quasi-isometric invariant of the graph $\Lambda_S$. We begin showing that F\o lner sets introduced in Theorem~\ref{def_ame} may be localized within $\mathcal{L}$-classes.
    \begin{lemma}\label{dom_meas_lclasses}
      Let $S$ be domain measurable. Then for every $\varepsilon > 0$ and finite $\mathcal{F} \subset S$, there exists some $\mathcal{L}$-class $L \subset S$ and a finite non-empty $F \subset L$ such that $|s (F \cap D_{s^*s}) \cup F| \leq (1 + \varepsilon) |F|$ for all $s \in \mathcal{F}$.
    \end{lemma}
    \begin{proof}
     The proof is essentially given in Lemma~5.2 of \cite{ALM19}. In fact, it is enough to note that by Lemma~\ref{lemma_action_lclasses} the equivalence relation $\approx$ used in~\cite{ALM19} is precisely Green's $\mathcal{L}$-relation.
    \end{proof}

    The following two lemmas allow us to reduce the proof of Theorem~\ref{thm_ame_qi} to the case of a surjective quasi-isometry.
    \begin{lemma}\label{dom_meas_onto}
      Let $S, T$ be quasi-isometric semigroups of bounded geometry. Then there is a finite group $G$ and a surjective quasi-isometry $\varphi \colon S \times G \rightarrow T$.
    \end{lemma}
    \begin{proof}
      Let $\phi \colon S \rightarrow T$ be a quasi-isometry. In particular, there is some $R > 0$ such that $d_T(t, \phi(S)) \leq R$ for every $t \in T$. Let $m > 0$ be an upper-bound of the cardinality of the $R$-balls in $T$, and let $G$ be any finite group of cardinality $m$. For any $t \in T$ let $\theta_t$ be an embedding of $B_R(t)$ into $G$ and consider the map
      \begin{equation}
        \varphi \colon S \times G \rightarrow T, \;\; \text{where} \;\; \varphi\left(s, g\right) :=
          \left\{
            \begin{array}{lc}
              \theta^{-1}_{\phi\left(s\right)}\left(g\right) & \; \text{if} \;\; g \in \text{im} \left(\theta_{\phi\left(s\right)}\right) \\
              \phi\left(s\right) & \; \text{otherwise}.
            \end{array}
          \right. \nonumber
      \end{equation}
      Then $\varphi$ is surjective and a local perturbation of $\phi$ and, thus, a quasi-isometry.
    \end{proof}
    \begin{lemma}\label{qi_inv_subgrp}
      Let $S$ be an inverse semigroup, and let $G$ be a finite group. If $S \times G$ is domain measurable (resp. amenable) then so is $S$.
    \end{lemma}
    \begin{proof}
      Observe the product in $S \times G$ is defined coordinate-wise. In particular $(s, g)^*(s, g) = (s^*s, 1)$ and, therefore, $(x, g) \in D_{(s^*s, 1)}$ if and only if $x \in D_{s^*s}$.

      Given $\varepsilon > 0$ and a finite $\mathcal{F} \subset S$ let $F_G \subset S \times G$ witness the $(\varepsilon/|G|, \mathcal{F} \times G)$-F\o lner condition  of $S \times G$ (cf., Theorem~\ref{def_ame}). Consider the set
      \begin{equation}
        F := \left\{ s \in S \mid \left(s, g\right) \in F_G \;\; \text{for some} \; g \in G \right\}. \nonumber
      \end{equation}
      Then $F \subset S$ is clearly finite and non-empty. Furthermore $|F| \cdot |G| \geq |F \times G|$ and, for any $s \in \mathcal{F}$
      \begin{equation}
        \frac{\left| s \left(F \cap D_{s^*s}\right) \setminus F\right|}{\left|F\right|} \leq \frac{\left| \left(s, 1\right) \left(F \times G \cap D_{\left(s^*s, 1\right)} \right) \setminus F \times G\right|}{\left|F \times G\right|} \, \cdot \, \left|G\right| \leq \frac{\varepsilon}{\left|G\right|} \, \cdot \, \left|G\right| \leq \varepsilon. \nonumber
      \end{equation}
      Therefore $F$ witnesses the $(\varepsilon, \mathcal{F})$-F\o lner condition of $S$ (cf., Theorem~\ref{def_ame}).

      If, in addition, $S \times G$ is amenable, then the set $F$ constructed before is contained in $D_{s^*s}$.
    \end{proof}

    The following proposition gives an alternative F\o lner type characterization of domain measurability in the case the pair $(S, K)$ admits a finite labeling (cf., Theorem~\ref{def_ame}). Given $R > 0$ and $A \subset S$ we denote by $\mathcal{N}_R A$ the \textit{$R$-neighborhood of $A$}, i.e., the set of points $x \in S$ such that $d(x, A) \leq R$. Note that, in particular, if $A \subset L$, where $L \subset S$ is an $\mathcal{L}$-class, then $A \subset \mathcal{N}_R A \subset L$.
    \begin{proposition}\label{fol_fl}
      Let $S = \langle K \rangle$ be an inverse semigroup such that $(S, K)$ admits a finite labeling. Then $S$ is domain measurable if and only if for every $R > 0$ and $\varepsilon > 0$ there is a finite non-empty $F \subset L \subset S$, where $L$ is an $\mathcal{L}$-class, such that $|\mathcal{N}_R F| \leq (1 + \varepsilon) |F|$.
    \end{proposition}
    \begin{proof}
      It is standard to show that if $S$ is domain measurable then the F\o lner condition in Theorem~\ref{def_ame} implies the corresponding inequality of the $R$-neighborhood (cf., \cite[Section~2]{ALLW18}). To show the reverse implication let $\varepsilon > 0$ and $R > 0$ be given, and let $\mathcal{F} \subset S$ be as in Proposition~\ref{prop:fl:newchr}. By Lemma~\ref{dom_meas_lclasses} there is an $\mathcal{L}$-class $L$ and a finite $F \subset L$  such that for every $s \in \mathcal{F}$
      \begin{equation}
        \frac{\left|s\left(F \cap D_{s^*s}\right) \cup F\right|}{\left|F\right|} \leq 1 + \frac{\varepsilon}{\left|F\right|}. \nonumber
      \end{equation}
      It follows that such an $F$ satiesfies the F\o lner type condition required.
    \end{proof}

    We can now generalize Theorem~\ref{ame_qiinv_gr} to the context of inverse semigroups.
    \begin{theorem}\label{thm_ame_qi}
      Let $S$ and $T$ be quasi-isometric inverse semigroups admitting finite labelings. If $T$ is domain measurable then so is $S$.
    \end{theorem}
    \begin{proof}
      Let $\varphi \colon S \rightarrow T$ be a quasi-isometry. By Lemma~\ref{dom_meas_onto} and Lemma~\ref{qi_inv_subgrp} we may suppose that $\varphi$ is a surjective quasi-isometry. Let $M > 0$ and $C \geq 0$ be some constants such that for every $x, y \in S$
      \begin{equation}
        \frac{1}{M} d_S\left(x, y\right) - C \leq d_T\left(\varphi\left(x\right), \varphi\left(y\right)\right) \, \leq \, M \, d_S\left(x, y\right) + C, \nonumber
      \end{equation}
      where $d_S$ and $d_T$ denote the path distances in $\Lambda_S$ and $\Lambda_T$, respectively. By Proposition~\ref{fol_fl} and given $R > 0$ and $\varepsilon > 0$ there is some finite non-empty $F_T \subset T$ within some $\mathcal{L}$-class and with small $(MR + C)$-neighborhood, that is, $|\mathcal{N}_{MR + C} F_T| \leq (1 + \varepsilon) |F_T|$. We claim the set
      \begin{equation}
        F_S := \varphi^{-1}\left(F_T\right) \nonumber
      \end{equation}
      has a small $R$-neighborhood in $S$ and, therefore, by Proposition~\ref{fol_fl}, $S$ is domain measurable too. Observe that $F_S$ is non-empty since $\varphi$ is surjective. Furthermore, the distance between two any $\mathcal{L}$-classes (in either $\Lambda_S$ or $\Lambda_T$) is infinite, and thus, since $\varphi$ is a quasi-isometry, it takes $\mathcal{L}$-classes of $S$ onto $\mathcal{L}$-classes of $T$. Therefore $F_S$ is contained within some $\mathcal{L}$-class. Moreover, $F_S$ must be finite since both $S$ and $T$ are of bounded geometry and $\varphi$ is a quasi-isometry. Finally:
      \begin{equation}
        \frac{\left|\mathcal{N}_{R} F_S\right|}{\left|F_S\right|} \leq \frac{\left|\mathcal{N}_{MR + C} F_T\right|}{\left|F_T\right|} \leq 1 + \varepsilon, \nonumber
      \end{equation}
      which proves that $F_S$ has small $R$-neighborhood in $S$.
    \end{proof}
    \begin{remark}
      Note that Theorem~\ref{thm_ame_qi} only proves that domain measurability is a quasi-isometry invariant. On the contrary, amenability, as considered by Day, is not. Indeed, let $S := \{1\} \sqcup \mathbb{F}_2$ and $T := \mathbb{F}_2 \sqcup \{0\}$, where $\mathbb{F}_2$ denotes the free non-abelian group of rank $2$. Then $S$ is non-amenable, while $T$ is. Moreover, the map from $S$ to $T$ sending $1$ to $0$ and $\mathbb{F}_2 \ni \omega \mapsto \omega \in \mathbb{F}_2 \subset T$ is an isometry, and hence also a quasi-isometry.
    \end{remark}

    %----------------------------------------------------------------------------------------
    % PROPERTY A AND SCHUTZENBERGER GRAPHS
    %----------------------------------------------------------------------------------------
    \subsection{Property A}\label{sec_scht_a} 
    Given an inverse semigroup $S$ recall first the construction of the $K$-labeled undirected graph $\Lambda_S=\sqcup_{e\in E(S)}\Lambda_{L_e}$ (see Section~\ref{sec_dist}), where the left-Schützenberger graphs $\Lambda_{L_e}$ correspond to the different connected components of $\Lambda_S$. We aim to characterize first when the left Sch\"{u}tzenberger graph $\Lambda_L$ of an $\mathcal{L}$-class $L \subset S$ has property A, generalizing the result for groups in Theorem~\ref{thm_gr_a}.

    To that end we first need to introduce the C*-algebraic notions of nuclearity and exactness. Both notions have distinct characterizations (see~\cite{BO08}) and have been studied extensively (see, for example, \cite{P12,M10,O00,O14,A16}). We will use the characterization in terms of contractive completely positive (c.c.p.) matrix approximations. A map between C*-algebras $\theta\colon \mathcal{A} \to \mathcal{B}$ is $\mathcal{A}$ is {\em nuclear} if for every $\varepsilon>0$ and every finite $\mathcal{F}\subset\mathcal{A}$ there exist c.c.p. maps $\varphi\colon \mathcal{A}\to M_n$ and $\psi\colon M_n\to \mathcal{B}$ such that $\|\psi\circ\varphi(A)-\theta(A)\|\leq \varepsilon$ for every $A\in\mathcal{F}$. A C*-algebra $\mathcal{A}$ is nuclear if the identity map $\mathrm{id}\colon\mathcal{A}\to\mathcal{A}$ is nuclear. A concretely represented C*-algebra $\mathcal{A} \subset \mathcal{B}\left(\mathcal{H}\right)$ is called \textit{exact} if the inclusion map $\iota\colon\mathcal{A}\to \mathcal{B}\left(\mathcal{H}\right)$ is nuclear.
    \begin{proposition}\label{prop_subalg_nuclear}
      Let $q > 0$ and let $X$ be a locally compact Hausdorff space. Every C*-subalgebra $\mathcal{A} \subset C_0(X) \otimes M_q$ is nuclear.
    \end{proposition}
    \begin{proof}
      Observe that every irreducible representation of $\mathcal{A}$ is of dimension, at most, $q$ and, thus, $\mathcal{A}$ is \textit{subhomogeneus} (see~\cite[Definition~2.7.6]{BO08}). Hence $\mathcal{A}$ is nuclear by~\cite[Proposition~2.7.7]{BO08}.
    \end{proof}
    
    The following result proves that for nuclearity it is enough to show that the identity map factors through nuclear algebras instead of matrices.
    \begin{proposition}\label{prop_nuclear_factor}
      A C*-algebra $\mathcal{R}$ is nuclear if and only if for every finite $\mathcal{F} \subset \mathcal{R}$ and $\varepsilon > 0$ there is a nuclear C*-algebra $\mathcal{A}$ and completely positive and contractive maps $\varphi \colon \mathcal{R} \rightarrow \mathcal{A}$ and $\psi \colon \mathcal{A} \rightarrow \mathcal{R}$ such that 
      $||\psi\circ\varphi(A)-A|| \leq \varepsilon$ for every $A \in \mathcal{F}$.
    \end{proposition}
    \begin{proof}
      The claim follows applying the nuclearity of $\mathcal{A}$ and an $\varepsilon/2$-argument.
    \end{proof}

    We prove now one of the main theorems of the section.
    \begin{theorem}\label{invsem_thm_l_a}
      Let $S = \langle K \rangle$ be a countable and discrete inverse semigroup, and let $L \subset S$ be an $\mathcal{L}$-class of bounded geometry. Let $p_L$ be the orthogonal projection from $\ell^2(S)$ onto $\ell^2(L)$. Consider the following statements:
      \begin{enumerate}
        \item \label{thm_l_a} The graph $\Lambda_L$ has property A.
        \item \label{thm_l_nuclear} The C*-algebra $p_L \mathcal{R}_S p_L$ is nuclear.
        \item \label{thm_l_r_exact} The C*-algebra $p_L \mathcal{R}_S p_L$ is exact.
        \item \label{thm_l_exact} The C*-algebra $p_L C^*_r(S) p_L$ is exact.
      \end{enumerate}
      Then (\ref{thm_l_a}) $\Rightarrow$ (\ref{thm_l_nuclear}) $\Rightarrow$ (\ref{thm_l_r_exact}) $\Leftrightarrow$ (\ref{thm_l_exact}). Moreover, if $(L, K)$ admits a finite labeling then all conditions are equivalent.
    \end{theorem}
    \begin{proof}
      For convenience, in this proof we denote by $p$ the projection $p_L$, and $\mathcal{R}$ stands for $\mathcal{R}_S$. Furthermore, note that $p V_s p = V_s p$, since $p$ projects onto an $\mathcal{L}$-class (see Lemma~\ref{lemma_l_classes}).

      (\ref{thm_l_a}) $\Rightarrow$ (\ref{thm_l_nuclear}). Let $\mathcal{F} \subset p \mathcal{R} p$ be finite and $\varepsilon > 0$. Observe that, without loss of generality, we may suppose that every element of $\mathcal{F}$ is of the form $f V_s$ where the support of $f$ is contained in $L \cap D_{ss^*}$. Let $R > 0$ be larger than the length of $s$ for every $s$ such that $f V_s \in \mathcal{F}$. Since $\Lambda_L$ has property A there are $C > 0$ and $\xi \colon \Lambda_L \rightarrow \ell^1(\Lambda_L)_{1, +}$ such that $\mathrm{supp}(\xi_x)\subset B_C(x)$ for every $x\in L$ and
      \begin{equation}\label{eq_l_a_eps}
        \left|\left| \, \xi_x - \xi_y \, \right|\right|_1 \leq \frac{\varepsilon}{M} \;\; \text{for every} \; x, y \in L \;\; \text{such that} \; d\left(x, y\right) \leq R,
      \end{equation}
      where $M := \max\left\{ \, \left|\left|f V_s\right|\right| \mid f V_s \in \mathcal{F} \, \right\}$. Consider then the map
      \begin{equation}
        \varphi \colon p \mathcal{R} p \rightarrow \prod_{x \in L} M_{B_C\left(x\right)}, \;\; \text{where} \; \varphi\left(a\right) := \left(p_{B_C\left(x\right)} a p_{B_C\left(x\right)}\right)_{x \in L}. \nonumber
      \end{equation}
      Recall that $M_{B_C(x)}$ denotes the full matrix algebra with rows and columns labeled by elements in $B_C(x)$ and $p_{B_C(x)} \in \ell^{\infty}(S)$ is the characteristic function of $B_C(x)$. Since $\Lambda_L$ is of bounded geometry there is some $q > 0$ such that $|B_R(x)| \leq q$ for every $x \in L$ and, therefore,
      \begin{equation}
        \text{im}(\varphi) \subset \prod_{x \in L} M_{B_C(x)} \subset \ell^{\infty}(L) \otimes M_q. \nonumber
      \end{equation}
      Let $\mathcal{A}$ be the closure in norm of the image of $\varphi$ and let
      \begin{equation}
        \psi \colon \mathcal{A} \rightarrow p \mathcal{R} p, \;\; \text{where} \; \psi\left(\left(b_x\right)_{x \in L}\right) := \sum_{x \in L} \xi_x^* b_x \xi_x
      \end{equation}
      where we identify $\xi_x$ with the diagonal operator $\xi_x \delta_y := \xi_y(x) \delta_y$ (note the flip between the argument and the index of $\xi$). The maps $\varphi$ and $\psi$ are c.c.p. and it follows that
      \begin{equation}
        \psi\left(\varphi\left(f V_s\right)\right) = \sum_{x \in L} f \xi_x V_s \xi_x = f V_s \left(\sum_{x \in L} \left(s^* \xi_x\right) \, \xi_x\right) \in p \mathcal{R} p, \nonumber
      \end{equation}
      where $s^* \xi_x (\delta_y) := \xi_{sy}(x)$ if $y \in D_{s^*s}$ and $s^* \xi_x (\delta_y) = 0$ otherwise. From Eq.~(\ref{eq_l_a_eps}) we have for all $s\in S$ such that 
      $f V_s \in \mathcal{F}$ the following estimate
      \begin{align}
        \left|\left| V_{s^*s} - \sum_{x \in L} \left(s^*\xi_x\right) \xi_x \right|\right| \, & = \, \sup_{y \in L} \left| V_{s^*s} \delta_y - \sum_{x \in L} \left(\left(s^*\xi_x\right) \xi_x\right) \delta_y \right| \nonumber \\
        & = \, \sup_{y \in L \cap D_{s^*s}} \left|1 - \sum_{x \in L} \xi_y\left(x\right) \xi_{sy}\left(x\right)\right| \, = \, \sup_{y \in L \cap D_{s^*s}} \left|1 - \langle \xi_{sy}, \xi_y\rangle \right| \leq \frac{\varepsilon}{M} \;.\nonumber
      \end{align}
      Therefore, we can estimate
      \begin{equation}
        \left|\left| \, f V_s - \psi\left(\varphi\left(f V_s\right)\right) \, \right|\right| = \left|\left| \, f V_s \left(V_{s^*s} - \sum_{x \in L} \left(s^*\xi_x\right) \, \xi_x\right) \, \right|\right| \, \leq \, \left|\left|f V_s\right|\right| \, \left|\left|V_{s^*s} - \sum_{x \in L} \left(s^*\xi_x\right) \xi_x\right|\right| \leq \varepsilon. \nonumber
      \end{equation}
      Since, by Proposition~\ref{prop_subalg_nuclear}, $\mathcal{A}$ is nuclear it follows that $p \mathcal{R} p$ is nuclear as well by Proposition~\ref{prop_nuclear_factor}.

      (\ref{thm_l_nuclear}) $\Rightarrow$ (\ref{thm_l_r_exact}) follows since nuclear algebras are exact.

      (\ref{thm_l_r_exact}) $\Rightarrow$ (\ref{thm_l_exact}) follows since exactness passes to subalgebras (see~\cite[Exercise~2.3.2]{BO08}).

      (\ref{thm_l_exact}) $\Rightarrow$ (\ref{thm_l_r_exact}). Given $\varepsilon > 0$ and a finite $\mathcal{F} \subset p \mathcal{R} p$, without loss of generality we may suppose again that every element in $\mathcal{F}$ is of the form $f V_s$, where the support of $f$ is contained in $L \cap D_{ss^*}$. By the exactness of $p C_r^*(S) p$ there are c.c.p. maps
      \begin{equation}
        \tilde{\varphi} \colon p C_r^*(S) p \rightarrow M_n \;\; \text{and} \;\; \tilde{\psi} \colon M_n \rightarrow \mathcal{B}\left(\ell^2\left(L\right)\right)\nonumber
      \end{equation}
      such that $||V_s p - \tilde{\psi}(\tilde{\varphi}(V_s p))|| \leq \varepsilon/M$ for all $f V_s \in \mathcal{F}$, where $M := \max\left\{ \, \left|\left|f V_s\right|\right| \mid f V_s \in \mathcal{F} \, \right\}$. Consider the c.c.p. maps
      \begin{align}
        \varphi \colon p \mathcal{R} p \rightarrow \ell^{\infty} \left(L\right) \otimes M_n & \quadtext{and} \psi \colon \ell^{\infty}\left(L\right) \otimes M_n \rightarrow \mathcal{B}\left(\ell^2\left(L\right)\right), \nonumber \\
        f V_s \; \rightarrow \; f \otimes \tilde{\varphi}\left(V_s\right) & \quadtext{and}  f \otimes b \rightarrow f \tilde{\psi}\left(b\right). \nonumber
      \end{align}
      Then
      \begin{align}
        \left|\left| \, \psi\left(\varphi\left(f V_s\right)\right) - f V_s \, \right|\right| \leq \left|\left|f V_s\right|\right| \, \cdot \, \left|\left| V_s p - \tilde{\psi}\left(\tilde{\varphi}\left(V_s p\right)\right) \right|\right| \leq \varepsilon, \nonumber
      \end{align}
      which proves, again using that $\ell^{\infty}(L) \otimes M_n$ is nuclear (see Proposition~\ref{prop_subalg_nuclear}), that $p \mathcal{R} p$ is exact. 

      Lastly, if the pair $(L, K)$ admits a finite labeling then by Theorem~\ref{thm_mtrspc_lclasses} the corner $p \mathcal{R} p$ is a uniform Roe algebra, i.e., $p \mathcal{R} p \cong C_u^*(\Lambda_L)$. Using recent results by Sako (see~\cite[Theorem~1.1]{S19}) we have that $C_u^*(\Lambda_L)$ is exact $\Leftrightarrow$ $C_u^*(\Lambda_L)$  is nuclear $\Leftrightarrow$  $\Lambda_L$ has property A.
    \end{proof}

    \begin{remark}
      Observe the implication (\ref{thm_l_a}) $\Rightarrow$ (\ref{thm_l_nuclear}) in Theorem~\ref{invsem_thm_l_a} is independent of $(L, K)$ admitting a finite labeling or not. This is remarkable for, in general, nuclearity does not pass to subalgebras and the corner $p_L \mathcal{R}_S p_L$ is properly contained in $C_u^*(\Lambda_L)$ if $(L, K)$ is not FL. In general, it is known that a metric space (say $\Lambda_L$) has property A if and only if its uniform Roe algebra (say $C_u^*(\Lambda_L)$) is nuclear (see, for example,~\cite[Theorem~5.5.7]{BO08}). Moreover,
      in Theorem~\ref{invsem_thm_l_a}~(\ref{thm_l_exact}) the corner $p_L C_r^*(S) p_L$ needs not be closed. Indeed, observe that, in general, $p_L$ is not contained in $C_r^*(S)$.
    \end{remark}
    Some examples of $\mathcal{L}$-classes with property A are every finite graph or every graph quasi-isometric to a tree (cf., \cite[Example~4.1.5]{NY12}). Nevertheless, contrary to the group case, an amenable inverse semigroup can have $\mathcal{L}$-classes without property A. Indeed, let $G$ be a group without property A (see, for example,~\cite{W08,O14}) and adjoin to it a zero element. Then $S := G \sqcup \{0\}$ is an amenable inverse semigroup where $G$ forms an $\mathcal{L}$-class without property A. We, however, would like to point out the following question, which is, in our opinion, more intriguing.
    \begin{enumerate}
      \item[\textbf{Q:}] Is there any class of inverse semigroups such that every/any of their Sch\"{u}tzenberger graphs are non-property A?
    \end{enumerate}

    We now extend Theorem~\ref{invsem_thm_l_a} to the whole graph $\Lambda_S=\sqcup_{e\in E(S)}\Lambda_{L_e}$. In particular, this allows us to characterize when the graph $\Lambda_S$ has property A via the C*-algebras $\mathcal{R}_S$ and $C_r^*(S)$. For similar results relating the amenability of the maximal homomorphic image $G(S)$ and the \textit{weak containment} of $S$ see~\cite[Proposition~4.1]{P78} and~\cite{M10}. For the relation between the nuclearity of $\mathcal{R}_S$ and the exactness or the nuclearity of $C_r^*(S)$ see~\cite{A16,A19}. We first note that $\Lambda_S$ has property A when all of its connected components have property A uniformly (see Definition~\ref{def_a}). 
    \begin{lemma}\label{lemma_unif_a}
      $\Lambda_S$ has property A if and only if the family $\{\Lambda_{L_e}\}_{e \in E(S)}$ has uniform property A.
    \end{lemma}
    \begin{proof}
     The proof is a direct consequence of the definitions involved.
    \end{proof}

    Observe that there are inverse semigroups, necessarily with infinitely many $\mathcal{L}$-classes, without property A and such that every of their Sch\"{u}tzenberger graphs do have property A (see Example~\ref{ex_boxspc_a}). We generalize next Theorem~\ref{invsem_thm_l_a} to $\Lambda_S$.
    \begin{theorem}\label{invsem_thm_a}
      Let $S = \langle K \rangle$ be a countable and discrete inverse semigroup of bounded geometry. Consider the assertions:
      \begin{enumerate}
        \item \label{thm_exact_a} The graph $\Lambda_S$ has property A.
        \item \label{thm_exact_nuclear} The C*-algebra $\mathcal{R}_S$ is nuclear.
        \item \label{thm_exact_r_exact} The C*-algebra $\mathcal{R}_S$ is exact.
        \item \label{thm_exact_exact} The C*-algebra $C^*_r(S)$ is exact.
      \end{enumerate}
      Then (\ref{thm_exact_a}) $\Rightarrow$ (\ref{thm_exact_nuclear}) $\Rightarrow$ (\ref{thm_exact_r_exact}) $\Leftrightarrow$ (\ref{thm_exact_exact}). Moreover, if the pair $(S, K)$ admits a finite labeling then all conditions are equivalent.
    \end{theorem}
    \begin{proof}
      By Lemma~\ref{lemma_unif_a}, condition (\ref{thm_exact_a}) is equivalent to the family $\{\Lambda_{L_e}\}_{e \in E(S)}$ having uniform property A. Observe that all the implications follow from similar arguments as those in the proof of Theorem~\ref{invsem_thm_l_a}. For instance, in this context the c.c.p. maps are given for each $\mathcal{L}$-class $L_e$, noting that the constants $C_e$, where $e\in E$, are uniformly bounded (see Definition~\ref{def_a}).
    \end{proof}

    From Theorem~\ref{invsem_thm_a} one arrives at the following immediate corollary. Recall that, given an inverse semigroup $S$, one may construct its \textit{universal groupoid} $G_\mathcal{U}(S)$, whose unit space is a locally compact totally disconnected Hausdorff space (see~\cite{P12,E08}). For the notion of amenability of groupoids we refer to
    \cite{AR00} (see also \cite{ES17}).
    \begin{corollary} \label{cor_ame_grpd}
      Let $S = \langle K \rangle$ be a countable and discrete inverse semigroup of bounded geometry. Furthermore suppose that $(S, K)$ admits a finite labeling. If Paterson's universal groupoid $G_\mathcal{U}(S)$ is amenable (as a groupoid) then $\Lambda_S$ has property A.
    \end{corollary}
    \begin{proof}
      Recall that Paterson's universal groupoid is amenable then its reduced algebra is nuclear. In turn, the reduced C*-algebra of $S$ and that of its universal groupoid coincide. Hence, by Theorem~\ref{invsem_thm_a}, the graph $\Lambda_S$ has property A.
    \end{proof}
    \begin{remark}
      Corollary~\ref{cor_ame_grpd} can be proven in an alternative way. Indeed, recall that the unit space $G_\mathcal{U}(S)^{(0)}$ of the universal groupoid of $S$ is the space of \textit{filters} of $E(S)$. The amenability condition of $G_\mathcal{U}(S)$ imposes a certain invariance condition on the sections of each filter and, as can be seen, the property A of $\Lambda_S$ is also an invariance condition on certain filters in $G_\mathcal{U}(S)^{(0)}$. In this way, the amenability of $G_\mathcal{U}(S)$ directly implies the property A of $S$, without needing to resort to C*-arguments.
    \end{remark}
    \begin{remark}
      Note that the reverse implication of Corollary~\ref{cor_ame_grpd} is false. Indeed, any non-abelian free group has property A, but is not amenable.
    \end{remark}

    The rest of the section aims to give a relation between the property A of the graph $\Lambda_S$ and that of $G(S)$ (see Proposition~\ref{prop_sa_gsa}). Before getting to the proof we need the next simple lemma (see~\cite[Lemma~4.26]{ALM19}).
    \begin{lemma}\label{invsem_lemma_projdec}
      Every countable inverse semigroup $S$ has a decreasing sequence of projections $\{e_n\}_{n \in \mathbb{N}}$ that is eventually below every other projection, that is, $e_n \geq e_{n + 1}$ and for every $f \in E(S)$ there is some $n_0 \in \mathbb{N}$ such that $f \geq e_{n_0}$.
    \end{lemma}
    \begin{proof}
      Since $S$ is countable we can enumerate the set of projections $E(S) = \left\{f_1, f_2, \dots\right\}$. The lemma follows putting $e_n := f_1 \dots f_n$.
    \end{proof}

    The proof of the following proposition is based on the facts that the left Cayley graph of $G(S)$ is the inductive limit of the Sch\"{u}tzenberger graphs of $S$ and that property A is closed under inductive limits with injective connecting maps. We give an explicit proof because, in general, the natural connecting maps $L_e \rightarrow L_{ef}$, where $xe \mapsto xef$, need not be injective for certain $\mathcal{L}$-classes.
    \begin{proposition}\label{prop_sa_gsa}
      Let $S$ be a countable and discrete inverse semigroup such that $\Lambda_S$ is of bounded geometry. If $\Lambda_S$ has property A, then so does the maximal homomorphic image $G(S)$.
    \end{proposition}
    \begin{proof}
      Fix a free ultrafilter $\omega \in \beta\mathbb{N} \setminus \mathbb{N}$ and let $\{e_n\}_{n \in \mathbb{N}} \subset E(S)$ be a decreasing sequence of projections as in Lemma~\ref{invsem_lemma_projdec}. In particular, observe that for all $s \in S$ we have that $s e_n \mathcal{L} e_n$ for $n$ large enough. Let $L_n \subset S$ be the $\mathcal{L}$-class of $e_n \in E(S)$. By Lemma~\ref{lemma_unif_a} the family $\{\Lambda_{L_n}\}_{n \in \mathbb{N}}$ has uniform property A and, thus, given $\varepsilon > 0$ and $R > 0$, let $\xi^{(n)} \colon \Lambda_{L_n} \rightarrow \ell^1(\Lambda_{L_n})_{1,+}$ witness the $(R, \varepsilon)$-uniform property A of 
      the family $\{\Lambda_{L_n}\}_{n \in \mathbb{N}}$. Consider
      \begin{equation}
        \xi \colon G\left(S\right) \rightarrow \ell^1\left(G\left(S\right)\right), \quad \text{where} \;\; \xi_{\sigma\left(s\right)}\left(\sigma\left(t\right)\right) := \lim_{n \rightarrow \omega} \xi^{\left(n\right)}_{se_n}\left(te_n\right). \nonumber
      \end{equation}
      Note that $\xi_{\sigma(s)}$ is a well defined positive function of norm $1$, since
      \begin{equation}
        \left|\left| \xi_{\sigma\left(s\right)} \right|\right|_1 = \sum_{\sigma\left(t\right) \in G\left(S\right)} \lim_{n \rightarrow \omega} \xi^{\left(n\right)}_{se_n}\left(te_n\right) = \lim_{n \rightarrow \omega} \sum_{t \in B_C\left(se_n\right)} \xi_{se_n}^{\left(n\right)}\left(te_n\right) = \lim_{n \rightarrow \omega} \left|\left| \xi^{\left(n\right)}_{se_n} \right|\right|_1 = 1. \nonumber
      \end{equation}
      In addition, for any $\sigma(s), \sigma(t) \in G(S)$ such that $\xi_{\sigma(s)}(\sigma(t)) \neq 0$, by the limit construction it follows that there is some positive $\delta > 0$ with $\xi_{se_n}^{(n)}(te_n) \geq \delta$ for all sufficiently large $n \in \mathbb{N}$. Therefore, using the comparison of distances given in Lemma~\ref{lemma_dist_lclasses}
      we obtain
      \begin{equation}
        d_{G\left(S\right)}\left(\sigma\left(s\right), \sigma\left(t\right)\right) \leq d_{L_n}\left(se_n, te_n\right) \leq C < \infty \nonumber
      \end{equation}
      where $C > 0$ is a constant bounding the diameter of the supports of $\xi_{se_n}^{(n)}$ (see Definition~\ref{def_a}). It thus follows that $\text{supp}(\xi_{\sigma(s)})$ is contained in a ball of radius $C > 0$ around $\sigma(s)$. Finally, let $\sigma(s), \sigma(t) \in G(S)$ such that $d_{G(S)}(\sigma(s), \sigma(t)) \leq R$. Then
      \begin{equation}
        \left|\left| \xi_{\sigma\left(s\right)} - \xi_{\sigma\left(t\right)} \right|\right|_1 = \lim_{n \rightarrow \omega} \left|\left| \xi_{se_n} - \xi_{te_n} \right|\right|_1 \leq \varepsilon \nonumber
      \end{equation}
      since $||\xi_{se_n} - \xi_{te_n}||_1 \leq \varepsilon$ for all sufficiently large $n \in \mathbb{N}$.
    \end{proof}

    %----------------------------------------------------------------------------------------
    % PROPERTY A IN THE E-UNITARY CASE
    %----------------------------------------------------------------------------------------
    \subsection{Property A in the E-unitary case}\label{sec_euni}
    We conclude the article applying Theorem~\ref{invsem_thm_a} to the special case where $S$ is \textit{E-unitary}. An inverse semigroup is \textit{E-unitary} if the relation generated by $\sigma$ and $\mathcal{L}$ is the equality, i.e., if $s \mathcal{L} t$ and $\sigma(s) = \sigma(t)$ then $s = t$ for every $s, t \in S$. Recall, for instance, that all F-inverse semigroups (see Section~\ref{sec_bnbprop_thm}) are E-unitary (cf., \cite[Proposition~3, Chapter~7]{L98}). The following proposition is an improvement of Proposition~\ref{prop_bddgeom}.
    \begin{proposition}\label{prop_euni_bddgeom}
      Let $S$ be an E-unitary inverse semigroup, and let $G(S)$ be its maximal homomorphic image. Then the following are equivalent:
      \begin{enumerate}
        \item \label{prop_euni_bddgeom_s} $\Lambda_S$ is of bounded geometry.
        \item \label{prop_euni_bddgeom_g} The left Cayley graph of $G(S)$ is of bounded geometry
      \end{enumerate}
    \end{proposition}
    \begin{proof}
      The implication (\ref{prop_euni_bddgeom_s}) $\Rightarrow$ (\ref{prop_euni_bddgeom_g}) is already proved in Proposition~\ref{prop_bddgeom}. To show the reverse implication
      let $R > 0$ and note that, by hypothesis, the $R$-balls of $G(S)$ have uniformly bounded cardinality, i.e., $\sup_{x \in S} |B_R(\sigma(x))| < \infty$. Since $S$ is E-unitary the canonical projection $\sigma$ gives an injective map from $B_R(x) \subset S$ to $B_R(\sigma(x)) \subset G(S)$ for every $x \in S$ and, thus,
      \begin{equation}
        \sup_{x \in S} \left|B_R\left(x\right)\right| \leq \sup_{x \in S} \left|B_R\left(\sigma\left(x\right)\right)\right| < \infty, \nonumber
      \end{equation}
      which proves that $\Lambda_S$ is of bounded geometry.
    \end{proof}
    Observe Proposition~\ref{prop_euni_bddgeom} (and Proposition~\ref{prop_bddgeom}) actually prove that any upper bound on the cardinality of the $R$-balls of $S$ is also an upper bound on the cardinality of the $R$-balls of $G(S)$. The reverse implication is also true if $S$ is, in addition, E-unitary. Moreover, note that the E-unitary assumption on $S$ is essential in Proposition~\ref{prop_euni_bddgeom}. Indeed, consider $S := G \sqcup \{0\}$, where $0$ denotes a zero element and $G$ is a countably generated group. Then $G(S)$ is trivial (hence of bounded geometry), but the left Sch\"{u}tzenberger graph of the $\mathcal{L}$-class of $1 \in G$ is the left Cayley graph of $G$ and, thus, not of bounded geometry.

    Recall that it was proven in~\cite[Proposition~8.5]{A16} that an E-unitary inverse semigroup $S$ is exact (in the sense that $C_r^*(S)$ is an exact C*-algebra) if and only if $G(S)$ is an exact group. With the techniques presented previously we can give a new proof of this result relating exactness of the C*-algebra to property A of the graph $\Lambda_S$.
    \begin{theorem}\label{thm_euni}
      Let $S = \langle K \rangle$ be an E-unitary countable and discrete inverse semigroup. Suppose $(S, K)$ admits a finite labeling. Then the following are equivalent:
      \begin{enumerate}
        \item \label{thm_euni_a} The graph $\Lambda_S$ has property A.
        \item \label{thm_euni_unif_a} The family of graphs $\{\Lambda_{L_e}\}_{e \in E(S)}$ has uniform property A.
        \item \label{thm_euni_exact} The reduced C*-algebra $C_r^*(S)$ is exact.
        \item \label{thm_euni_g} The maximal homomorphic image $G(S)$ has property A.
      \end{enumerate}
    \end{theorem}
    \begin{proof}
      The equivalences (\ref{thm_euni_a}) $\Leftrightarrow$ (\ref{thm_euni_unif_a}) and (\ref{thm_euni_a}) $\Leftrightarrow$ (\ref{thm_euni_exact}) were proven in Lemma~\ref{lemma_unif_a} and Theorem~\ref{invsem_thm_a}, respectively. Moreover, the implication (\ref{thm_euni_a}) $\Rightarrow$ (\ref{thm_euni_g}) is proved in Proposition~\ref{prop_sa_gsa}, and, thus, it suffices to show (\ref{thm_euni_g}) $\Rightarrow$ (\ref{thm_euni_exact}). Moreover, from \cite[Theorem~4.4.2]{P12} we have that $C_r^*(S) = C_0(X) \rtimes_r G(S)$, where $X = G_\mathcal{U}(S)^{(0)}$ is the unit space of the universal groupoid. Indeed, it is routine to show that if $S$ is E-unitary then $G_\mathcal{U}(S)$ is actually the transformation groupoid given by $G(S)$ acting on the space of filters of $E(S)$. Thus, following similar reasonings as in the proof of the same implication of Theorems~\ref{invsem_thm_l_a} and \ref{invsem_thm_a}, it follows that if $G(S)$ is 
      exact then so is $C_r^*(S)$.
    \end{proof}

    We conclude the article with a remark and a question in relation to Theorem~\ref{thm_euni}.
    \begin{remark}\label{rem_sigma_not_qi}
      The proof of the preceding theorem shows that the property A of $\Lambda_S$ and that of $G(S)$ are strongly related in the E-unitary case. However, the proof is indirect (in the sense that it is based on C*-properties). Moreover, it is well known (to experts) that, in general, the canonical projection $\sigma \colon S \rightarrow G(S)$ needs not coarsely embed any Sch\"{u}tzenberger graph of $S$ into $G(S)$, even in the finitely generated and E-unitary case. This suggests the following question.
      \begin{itemize}
       \item[\textbf{Q:}] Is there a direct (coarse) geometric technique translating property A of the maximal homomorphic image $G(S)$ into property A of any Sch\"{u}tzenberger graph of $S$ in the E-unitary case?
      \end{itemize}
    \end{remark}

  %----------------------------------------------------------------------------------------
  % BIBLIOGRAPHY: format \bibitem{label} S.~Name, \textit{Title}, Journal. \textbf{volume} (year) page--page.
  %----------------------------------------------------------------------------------------
  \providecommand{\bysame}{\leavevmode\hbox to3em{\hrulefill}\thinspace}
  
\end{document}